\documentclass[a4paper,12pt]{amsart}
\usepackage{hyperref}
\usepackage[margin=1in]{geometry}
\usepackage{amsmath, amsthm}
\usepackage{amssymb}
\usepackage{graphicx}                                          % Allows us to include images
\usepackage{float}                                                % Allows proper centering of images
\usepackage{caption}                                            % Allows captions on images
\usepackage{cancel}                                             % Allows "crossing out"
\usepackage{soul}                                                  % Allows highlighting: \hl{}
\usepackage{xcolor}
\usepackage{enumerate}
\usepackage{thmtools, thm-restate}
\usepackage{url}

\theoremstyle{plain}
\newtheorem{theorem}{Theorem}[section]
\newtheorem{lemma}[theorem]{Lemma}

\newtheorem{corollary}[theorem]{Corollary}

\newtheorem{definition}[theorem]{Definition}

\newtheorem{example}[theorem]{Example}

\numberwithin{equation}{section}

\newcommand{\eps}{\varepsilon}
\DeclareMathOperator{\Var}{Var}
\newcommand{\Ex}{\mathbb{E}}

\begin{document}
	
	\title{Super-logarithmic cliques in dense inhomogeneous random graphs}
	\date{\today}
	\author{Gweneth McKinley}
	\maketitle
	\begin{abstract}
		In the theory of dense graph limits, a \emph{graphon} is a symmetric measurable function $W\colon[0,1]^2\to [0,1]$. Each graphon gives rise naturally to a random graph distribution, denoted $\mathbb{G}(n,W)$, that can be viewed as a generalization of the Erd\H{o}s-R\'enyi random graph.	Recently, Dole{\v{z}}al, Hladk{\'y}, and M{\'a}th{\'e} gave an asymptotic formula of order $\log n$ for the clique number of $\mathbb{G}(n,W)$ when $W$ is bounded away from 0 and 1. We show that if $W$ is allowed to approach 1 at a finite number of points, and displays a moderate rate of growth near these points, then the clique number of $\mathbb{G}(n,W)$ will be $\Theta(\sqrt{n})$ almost surely. We also give a family of examples with clique number $\Theta(n^\alpha)$ for any $\alpha\in(0,1)$, and some conditions under which the clique number of $\mathbb{G}(n,W)$ will be $o(\sqrt{n})$, $\omega(\sqrt{n}),$ or $\Omega(n^\alpha)$ for $\alpha\in(0,1)$.
	\end{abstract}
	
	\section{Introduction}
	
	The Erd\H{o}s-R\'enyi random graph $G_{n,p}$ is a graph on $n$ vertices where an edge is placed independently with probability $p$ between each pair of vertices. Since its introduction in 1959 by Gilbert \cite{Gi59} and by Erd\H{o}s and R\'enyi \cite{ErRe59}, it has become one of the fundamental objects of study in probabilistic combinatorics, and a wide variety of its properties are well understood. One of the most basic parameters of any graph $G$ is the \emph{clique number} $\omega(G)$, the number of vertices in the largest complete subgraph of $G$. It was shown independently by Grimmett and McDiarmid in 1975 \cite{GrMc75} and Matula in 1976 \cite{Ma76} that for a fixed $p\in (0,1)$, the clique number $\omega(G_{n,p})$ of $G_{n,p}$ satisfies
	\begin{equation}\label{eqn:ERCliqueNum}
	\omega(G_{n,p}) = (1+o(1))\cdot\frac{2\log n}{\log(1/p)}
	\end{equation}
	\noindent with probability $1-o(1)$ as $n$ approaches infinity. This can be proved roughly as follows: we obtain an upper bound on $\omega(G_{n,p})$ by finding $k$ such that the expected number of $k$-cliques in $G_{n,p}$ is asymptotically zero (the first moment method). Then, to prove a matching lower bound, we show that for an appropriate, slightly smaller $k$, the number of $k$-cliques in $G_{n,p}$ approaches infinity in the limit and has low variance. This implies that the number of cliques of size $k$ is highly concentrated around its expectation, and will be positive with high probability (the second moment method). Some variation on this method has been a standard technique for computing clique number in other random graph models as well. (See \cite{GrMc75}, \cite{Ma76}, \cite{DGLU11}, \cite{DHM17}, and \cite{BCV18}.) 
	
	The Erd\H{o}s-R\'enyi random graph $G_{n,p}$ may be considered ``homogeneous" in the following sense: between every pair of vertices, an edge is assigned independently with the \emph{same} probability $p$. In recent years, interest has been developing in studying inhomogeneous random graphs; in this model, edges are assigned between some pairs with higher or lower probabilities. This is both a better model of many real-world phenomena and an object of independent mathematical interest. However, with this greater flexibility comes greater difficulty in analysis. In this paper, we will characterize the clique numbers of a variety of inhomogeneous random graphs that arise from the theory of (dense) graphons. 
	
	A \emph{graphon} $W$ is defined as a symmetric, measurable function from $\Omega^2$ to $[0,1]$, where $\Omega$ is a probability space. To obtain a random graph from $W$, we sample $n$ points $x_1,\dots, x_n$ independently according to the probability distribution on $\Omega$, and connect vertices $i$ and $j$ by an edge with probability $W(x_i,x_j)$, independently for each pair $(i,j)$. (In this paper, for the sake of brevity, we will often identify the vertex $i$ with the value $x_i$, and speak of ``sampling vertices" from $\Omega$). We denote this graph by $\mathbb{G}(n,W)$, and refer to it as a ``$W$-random graph". Notice that in the case where $W$ is equal to the constant function $p$, we simply have $\mathbb{G}(n,W) = G_{n,p}$. One of the main results in the theory of graphons, proved by Lov{\'a}sz and Szegedy in 2006 in \cite{LoSz06}, is that every infinite sequence of graphs contains a subsequence converging to some graphon $W$ (in what is called the cut norm), and moreover, that every graphon can be achieved in this way, as the limit of some sequence of graphs. It is therefore reasonable to think of graphons as the correct limiting objects for sequences of graphs that are Cauchy sequences in an appropriate metric. See \cite{Lo12} for a detailed survey of the theory of graphons.

	It should be noted that we must take some care in defining a notion of clique number for graphons. We might hope that all sequences of graphs converging to a given graphon would have the same clique number asymptotically; however, as noted in \cite{DHM17}, this is not the case. Consider as an example the following two sequences of graphs.\\
	
	\noindent\begin{minipage}{1\linewidth}
		\begin{example} \mbox{}
			\begin{itemize}
				\item $G_n$ consists of a clique on $\sqrt{n}$ vertices, and $n-\sqrt{n}$ isolated vertices.
				\item $H_n$ consists of $n$ isolated vertices.\\
			\end{itemize}
		\end{example}
	\end{minipage}
	\indent Both sequences approach the zero graphon, as the density of edges approaches zero in both cases. However $\omega(G_1) = \sqrt{n}$, while $\omega(H_1) = 0$. Thus, instead of looking at all sequences of graphs converging to a given graphon $W$, we will consider only ``typical" sequences, sampled according to the distribution $\mathbb{G}(n,W)$. (Note: an alternate notion of clique number for a graphon is presented in \cite{HlRo17}.)
	
	This was the question considered by Dole{\v{z}}al, Hladk{\'y}, and M{\'a}th{\'e} in \cite{DHM17}, where they obtained a partial characterization of the clique number of $\mathbb{G}(n,W)$ for graphons $W$. They proved the following result. (In the statement below, by ``essentially bounded", we mean that the given bound holds everywhere except perhaps on some set of measure zero.)
	
	\begin{theorem}[Dole{\v{z}}al, Hladk{\'y}, and M{\'a}th{\'e} {\cite[Cor. 2.8]{DHM17}}]\label{thm:DolezalCliqueNum}
		For a graphon $W\colon\Omega^2\to [0,1]$ that is essentially bounded away from 0 and 1, 
		\[
		\omega(\mathbb{G}(n,W)) = (1+o(1))\kappa(W)\log n,
		\]
		asymptotically almost surely, where
		\[
		\resizebox{1\hsize}{!}{$\kappa(W) = \sup\left\{ \frac{2\| h\|_1^2}{\int_{(x,y)\in\Omega^2}h(x)h(y)\log(1/W(x,y))\,d(\nu^2)}
			\, :\, \text{h is a nonnegative }L^1\text{-function on }\Omega
			\right\}.$}
		\]
	\end{theorem}
	Notice that, for a graphon $W$ essentially bounded between $p_1>0$ and $p_2<1$, we can couple $\mathbb{G}(n,W)$ with the Erd\H os-R\'enyi random graphs $G_{n, p_1}$ and $G_{n, p_2}$ so that $\omega(G_{n, p_1}) \leq \omega(\mathbb{G}(n,W))\leq \omega(G_{n, p_1})$. Since the clique number of $G_{n,p}$ is $\Theta(\log n)$ for any value of $p$, this immediately tells us that the clique number of $\mathbb{G}(n,W)$ is also $\Theta(\log n)$ with probability approaching 1. Thus the key part of the result above is the characterization of the constant $\kappa(W)$ in $\Theta(\log n)$.
	
	A similar question was considered by Bogerd, Castro, and van der Hofstad in \cite{BCV18}; they studied clique number for rank-1 inhomogeneous random graphs, in which a graph is formed by assigning a weight to each vertex according to some distribution, and then connecting each pair of vertices independently with a probability proportional to the product of their weights. They showed that, if all vertex weights are bounded away from 1 (analogous to the assumption in Theorem~\ref{thm:DolezalCliqueNum} that $W$ is essentially bounded away from 1), then the clique number of such a graph is concentrated on at most two consecutive integers, for which they gave explicit expressions. This was proved in both the dense case and the sparse case, in which the edge density approaches zero as the number of vertices grows. It should be noted that a great deal of the work on inhomogeneous random graph models has centered on the sparse case, which gives a more accurate model for a variety of real-world networks, and it would  be interesting to see more results in this direction.  (See \cite{BJR07} for one of the seminal sparse models, and \cite{Va16} and \cite{Va18} for a survey of other recent work.) Results have also been obtained for clique number in random graphs with a power-law distribution \cite{JLN10} and hyperbolic random graphs \cite{BFK18}.
	
	Here, however, we will explore in a different (and in some sense, even opposite) direction. Namely, for graphons $W$ that are not bounded away from 1, even the rough order of growth of $\omega(\mathbb{G}(n,W))$ is not apparent (we could think of this as producing a $W$-random graph with potentially very dense spots); for this reason, it is interesting to ask what may happen if $W$ is allowed to approach 1. (Note, however, that if $W=1$ on $S\times S$ for some set $S$ of positive measure, then $W$ will have linear clique number, as the subset of vertices sampled from $S$ will all be connected with probability 1.) Additionally, although the restriction to graphons essentially bounded away from 1 given in \cite{DHM17} is a natural condition that precludes a variety of pathological examples, there is no reason to suppose that any particular graphon that might arise in an applied setting would necessarily satisfy it. It is still necessary, however, to impose some restrictions on the behavior of $W$ in order to obtain a good characterization of $\omega(\mathbb{G}(n,W))$; the authors of \cite{DHM17} also showed that for an arbitrary graphon $W$ not bounded away from 1, $\omega(\mathbb{G}(n,W))$ may behave quite wildly as $n\to\infty$.
	\begin{example}[Dole{\v{z}}al, Hladk{\'y}, and M{\'a}th{\'e} {\cite[Prop. 2.1]{DHM17}}]
		There exists a graphon $W$ and a sequence of integers $n_1<n_2<\cdots$ such that, asymptotically almost surely, $\omega(\mathbb{G}(n_i,W))$ alternates between at most $\log\log n_i$ and at least $\frac{n_i}{\log\log n_i}$ on elements of the sequence.
	\end{example}
	In fact, we may take any $\omega(1)$ function in place of $\log\log n$ in the example above. This behavior is shown in \cite{DHM17} to be achieved by a highly discontinuous graphon \mbox{$W\colon[0,1]^2\to [0,1]$}, which raises the question: even if $W$ is not bounded away from $1$, can we obtain a good characterization of $\omega(\mathbb{G}(n,W))$ as long as $W$ is reasonably well-behaved?	This is the central question of this paper. In order to characterize a graphon as well-behaved, we would like to have some notion of continuity, smoothness, etc. With this in mind, for the majority of this paper, we will restrict ourselves to graphons on $[0,1]^2$, as opposed to $\Omega^2$ for a more general probability space $\Omega$. Although in some applications, it may be more appropriate to work in a more general space $\Omega$, it is unclear what level of generality would allow us as much freedom as possible in choosing the space $\Omega$ while still capturing appropriate notions of ``well-behavedness". Thus, for brevity and clarity of analysis, we will limit ourselves to $[0,1]^2$, which serves as a good illustrative case of all the ideas here.  We also note that, for a graphon $W\colon[0,1]^2\to[0,1]$, among points with $W(x,y) = 1$, we are primarily concerned with those points along the line $x=y$, as shown by the following lemma.
	
	\begin{restatable}{lemma}{DiagonalLemma}
		\label{lem:TheImportanceOfBeingDiagonal}
		Let $W\colon[0,1]^2\to[0,1]$ be a graphon whose essential supremum is strictly less than 1 in some neighborhood of each point $(x,x)$ for $x\in[0,1]$. Then $\omega(\mathbb{G}(n,W)) = O(\log n)$ asymptotically almost surely.
	\end{restatable}
	
	This lemma is proved in Section \ref{sec:Prelim}. With this in mind, our goal is really to find the clique number associated to a ``well-behaved" graphon that is equal to $1$ at one or more points $(x,x)$ with $x\in[0,1]$. The main contribution of this paper consists of several such results. Before presenting these results, however, one final observation: it is perhaps natural to ask whether graphons that are close in cut distance will produce $W$-random graphs whose clique numbers are close asymptotically. In general, however, this is not the case. This can be illustrated by a wide variety of examples, but perhaps the simplest is the following family of graphons on $[0,1]^2$: 
	\[
	W_\eps(x,y) = \begin{cases}
	1 &\text{if }(x,y)\in[0,\eps]^2\\
	0 & \text{otherwise}
	\end{cases}
	\]
	for each $\eps >0$. Under the cut norm, $W_\eps$ converges to the zero graphon as $\eps\to 0$, but for any fixed $\eps$, the clique number of $\mathbb{G}(n,W_\eps)$ is $(1+o(1))\eps n = \Theta(n)$ asymptotically almost surely (see Lemma~\ref{lem:IntervalPointConcent}). Indeed, the primary driver of clique number for a $W$-random graph is not global behavior (as measured by the cut norm), but local behavior near points where $W$ is maximized. Following are several results characterizing clique number in terms of this local behavior. First, and perhaps surprisingly, for a graphon equal to 1 at only a finite number of points $(x,x)$, we will very often obtain a clique number of $\Theta(\sqrt{n})$.
	
	\begin{theorem}\label{thm:DirectionalDerivativeCliqueNum}
		Let $W\colon[0,1]^2\to [0,1]$ be a graphon equal to $1$ at some collection of points $(a_1,a_1),\dots, (a_k,a_k)$, and essentially bounded away from 1 in some neighborhood of $(x,x)$ for each other $x\in[0,1]$. If all directional derivatives of $W$ exist at the points $(a_1,a_1),\dots, (a_k,a_k)$, and are uniformly bounded away from 0 and $-\infty$, then $\omega(\mathbb{G}(n,W)) = \Theta(\sqrt{n})$ asymptotically almost surely.
	\end{theorem}
	
	We can expand this result to graphons $W$ whose directional derivatives are not defined at the points where $W(a,a) = 1$. In Section \ref{sec:Sqrt}, we give a more complete characterization in terms of the Dini derivatives of $W$ (the limsup and liminf of the difference quotient that defines the ordinary derivative) at the points where $W(a,a) = 1$. In particular, this characterization will show that if $W$ is ``too steep" at the points where it is equal to 1, then the clique number of $\mathbb{G}(n,W)$ will be $o(\sqrt{n})$ (Lemma~\ref{lem:UpperDini} (\ref{lem:UpperDiniInfty})), and if $W$ is ``too flat" at these points  (derivatives equal to zero), then the clique number will be $\omega(\sqrt{n})$ (Lemma~\ref{lem:LowerDini} (\ref{lem:LowerDiniZero})). This expanded characterization will also yield the following. 
	
	\begin{restatable}{lemma}{firstLipschitz}
		\label{lem:Lipschitz}
		Let $W\colon[0,1]^2\to [0,1]$ be a graphon equal to 1 at some point $(a,a)$. If $W$ is locally Lipschitz continuous at $(a,a)$, then $\omega(\mathbb{G}(n,W)) = \Omega(\sqrt{n})$ asymptotically almost surely.
	\end{restatable}
	
	(We will recall the definition of local Lipschitz continuity at a point $(a,a)$ in the proof of Lemma~\ref{lem:Lipschitz}.) In addition, in Section \ref{sec:Poly}, we present a family of graphons yielding clique numbers $\Theta(n^\alpha)$ for any constant $\alpha>0$. 
	
	\begin{restatable}{theorem}{PolyFamClique}
		\label{thm:PolyFamCliqueNum}
		For any constant $r>0$, define the graphon
		\[U_r(x,y) := (1-x^r)(1-y^r).\]
		The random graph $\mathbb{G}(n,U_r)$ asymptotically almost surely has clique number $\Theta(n^{\frac{r}{r+1}})$.
	\end{restatable}
	
	It will be shown in Section~\ref{sec:Poly} that this implies the following more general result.
	
	\begin{restatable}{theorem}{AlphaHolder}\label{thm:AlphaHolder}
		Let $W\colon[0,1]^2\to [0,1]$ be a graphon equal to 1 at some point $(a,a)$. If $W$ is locally $\alpha$-H\"older continuous at $(a,a)$ for some constant $\alpha$, then $\omega(\mathbb{G}(n,W)) = \Omega(n^{\frac{\alpha}{\alpha+1}})$ asymptotically almost surely.
	\end{restatable}
		
	(We will recall the definition of local $\alpha$-H\"older continuity at a point $(a,a)$ immediately before the proof of Theorem~\ref{thm:AlphaHolder}.) We will also be able to use the characterization of $\omega(\mathbb{G}(n,U_r))$ given in Section~\ref{sec:Poly} to show that if a graphon $W$ has infinitely many derivatives equal to zero at a point $(a,a)$ where $W(a,a) = 1$, then the clique number of $\mathbb{G}(n,W)$ will be $n^{1-o(1)}$ asymptotically almost surely. In other words, if $W$ is ``extremely flat" at the points where it is equal to 1, then the clique number of $\mathbb{G}(n,W)$ will be nearly linear. We will prove this for the following specific example, but the same reasoning can apply more generally.
	
	\begin{restatable}{proposition}{TaylorSeriesAlmostLinear}\label{prop:TaylorSeriesAlmostLinear}
		For the graphon $W:[0,1]^2\to [0,1]$ defined by
		\[W=(1-f(x))(1-f(y)),\text{ where } f(x) = 
		\begin{cases}
		e^{-1/x^2} & x\neq 0\\
		0 & x = 0
		\end{cases},\]
		the random graph $\mathbb{G}(n,W)$ has clique number $n^{1-o(1)}$ asymptotically almost surely. 
	\end{restatable}
	
	It should be noted that, in contrast to Theorem~\ref{thm:DolezalCliqueNum} and the characterization of clique number for Erd\H os-R\'enyi random graphs, all the results above give the relevant clique number up to a constant. For any graphon $W$, however, the clique number of $\mathbb{G}(n,W)$ is highly concentrated for large values of $n$; the following was proved in \cite{DHM17}. (This is slightly different from the original formulation, but follows directly the proof of Theorem 2.2 in \cite{DHM17}.)
	
	\begin{theorem}[Dole{\v{z}}al, Hladk{\'y}, and M{\'a}th{\'e} {\cite[Thm. 2.2]{DHM17}}]\label{thm:DolezalConcentration}
		For any graphon $W$, with probability $1-o(1)$,
		\[
		\omega(\mathbb{G}(n,W)) = (1+o(1))\cdot \mathbb{E}[\omega(\mathbb{G}(n,W))].
		\]
	\end{theorem}
	
	From this, we know that a correct constant always exists. But although the clique number $\omega(\mathbb{G}(n,W))$ is almost always very close to its expectation, it may occasionally be very large. Indeed, for many of the examples we will consider, the number of cliques of an appropriate size will have quite high variance, making it impossible to directly apply the second moment method as with Erd\H os-R\'enyi random graphs. (This obstacle is detailed more explicitly in Sections \ref{sec:Sqrt} and \ref{sec:Poly}, with proofs given in the appendix.) Instead, in proving the results above, we use the first moment method to establish upper bounds (the standard technique), while for lower bounds, we directly attempt to predict which vertices are likely to form a large clique, and show that this does indeed happen with high probability. It seems likely that to improve these lower bounds, a different technique would be necessary.
	
	It should be noted that the authors of \cite{DHM17} did not use the second moment method directly to prove Theorem~\ref{thm:DolezalCliqueNum}, but instead applied it to a carefully selected restriction of the graphon $W$, converting this back into a lower bound on $\omega(\mathbb{G}(n,W))$ via a somewhat complex argument. It is possible that a similar techinque could be used to improve some or all of the lower bounds given here. It is also possible that tighter bounds could obtained using techniques from the theory of large deviations as in \cite{AP03}; in this case, instead of looking at the (random) number of cliques $X$ of a given size in $\mathbb{G}(n,W)$ and attempting to give upper and lower bounds on $X$ that hold with high probability, we would define a random variable $X'$ that gives greater weight to those cliques arising from a ``typical" configuration of vertices (e.g., not too many vertices sampled from a small interval), and that would thus have lower variance than $X$. If we could find upper and lower bounds on $X'$, these could then be translated into upper and lower bounds on $X$.
	
	Another potentially interesting extension of the results above could be to consider graphons with an infinite number (either countable or uncountable) of points $(x,x)$ with $W(x,x) = 1$. For example, the following graphon is equal to 1 along the line $x=y$ and drops away from 1 off that line.
	
	\begin{restatable}{proposition}{LineExample}
		\label{prop:LineExample}
		Let $W(x,y) = 1-|x-y|$. The clique number of $\mathbb{G}(n,W)$ is $n^{1/2+o(1)}$ asymptotically almost surely.
	\end{restatable}
	
	This will be proved in Section \ref{sec:Extensions}; the lower bound follows directly from applying Lemma~\ref{lem:Lipschitz} to any point on the line $x=y$, and the upper bound is a fairly straightforward calculation. Both arguments could be used on a wide variety of such examples, but both are likely not tight in general. 
	
	As a last note, we discuss briefly the related problems of finding a large clique or a planted clique in a random graph, and how they relate to the work here. It is a long-standing problem, proposed by Karp in 1976 \cite{Ka76} to find a clique of size $(1+\eps)\log_2(n)$ in $G_{n,1/2}$ in polynomial time. (A clique of size $\sim 2\log_2(n)$ will almost always exist.) There are several polynomial-time algorithms that find a clique of size $(1+o(1))\log_2(n)$ (e.g., \cite{KS98}), but the original problem remains open. Of a similar flavor, but slightly different, the planted clique problem asks us to find a clique of size $k$ that is ``planted" in an Erd\H{o}s-R\'enyi random graph $G_{n,p}$ by randomly selecting $k$ vertices and adding all possible edges between them; we may ask for an algorithm that runs either in polynomial or unbounded time. In unbounded time, the planted clique can be recovered for $k$ quite close to the expected clique number for $G_{n,p}$, but perhaps surprisingly, if we ask for a polynomial-time algorithm, the best known methods find the planted clique only for $k=c\cdot\sqrt{n}$, for some particular constant $c$ (first proved in \cite{AKS98}, with a variety of simpler algorithms or algorithms improving the constant found later; see, for example, \cite{FR10} and \cite{DM15}). It could be interesting to explore these problems in the setting where the background graph is inhomogeneous (as opposed to $G_{n,p}$); it seems entirely possible that a large clique or hidden clique could be easier to recover in this setting. Indeed, this has been shown to be the case for several specific (mostly sparse) random graph models (see \cite{FK12}, \cite{BFK18}, and \cite{JLN10}). It is possible that a more general result along these lines could be established for some of graphs discussed here, or those in \cite{DHM17} or \cite{BCV18}, especially given knowledge of the clique number in the background graph.
	
	The remainder of this paper is structured as follows. In Section \ref{sec:Prelim}, we prove Lemma~\ref{lem:TheImportanceOfBeingDiagonal} and a few other simple technical lemmas that will be used throughout the paper. In Section \ref{sec:Sqrt}, we present a family of graphons giving clique numbers $\Theta(\sqrt{n})$, and use this to prove Theorem~\ref{thm:DirectionalDerivativeCliqueNum}, Lemma~\ref{lem:Lipschitz}, and an extension to graphons satisfying a more general set of conditions. In Section \ref{sec:Poly}, we prove Theorem~\ref{thm:PolyFamCliqueNum},  Theorem~\ref{thm:AlphaHolder}, and Proposition~\ref{prop:TaylorSeriesAlmostLinear}. In Section \ref{sec:Extensions}, we prove Proposition~\ref{prop:LineExample}, and discuss possible extensions of this work. And finally, in the appendix, we prove that for many of the $W$-random graphs discussed in other parts of the paper, the number of cliques of an appropriate size has high variance, making a direct application of the second moment method to establish a lower bound on the clique number impossible in those cases.
	
	\section{Preliminaries}\label{sec:Prelim}
	
	In this section, we establish some notation and technical lemmas that will be used throughout the rest of the paper. We will often want to focus only on a small portion of a graphon $W$, typically a neighborhood around a point where $W$ is equal to 1. In order to analyze how local behavior affects the clique number of a $W$-random graph, we first ascertain how many vertices will typically be sampled from a given neighborhood. Note: we will use ``a.a.s." throughout as an abbreviation for ``asymptotically almost surely", i.e., with probability approaching 1 as $n$ approaches $\infty$. And below, we write $\lambda$ for the Lebesgue measure on $\mathbb{R}$.
	
	\begin{lemma}\label{lem:IntervalPointConcent}
		Let $A(1),A(2),\dots$ be measurable subsets of $[0,1]$ with $\lambda(A(n))=\omega\left(\frac{1}{n}\right)$. Among $n$ points uniformly distributed on the interval $[0,1]$, the number of points in $A(n)$ will a.a.s.\ be $(1+o(1))n\lambda(A(n))$.
	\end{lemma}
	
	\begin{proof}
		The number of points $X$ in any given subset of $[0,1]$ of measure $\lambda(A(n))$ is a binomial random variable with parameters $n$ and $p=\lambda(A(n))$. Therefore
		\begin{align*}
		\mathbb{E}[X]= np
		\end{align*}
		and
		\begin{align*}
		\operatorname{Var}(X) &= np(1-p)\\
		&\leq np.
		\end{align*}
		Thus, for any $\eps>0$, by Chebyshev's inequality
		\begin{align*}
		\Pr\left[|X-\mathbb{E}[X]|\geq \eps\mathbb{E}[X]\right]\ \leq\ \frac{\Var(X)}{\eps^2\mathbb{E}[X]^2}
		\ \leq\ \frac{np}{\eps^2(np)^2}
		\ =\ \frac{1}{\eps^2np}.
		\end{align*}
		By assumption, $np = n\cdot \lambda(A(n)) = \omega(1)$. So taking, for instance, $\eps^2 = (np)^{-1/2}$, we have
		\begin{align*}
		\Pr\left[|X-\mathbb{E}[X]|\geq \tfrac{1}{(np)^{1/4}}\cdot \mathbb{E}[X]\right]\leq\tfrac{1}{(np)^{1/2}}=o(1).
		\end{align*}
		Thus with probability $1-o(1)$, we have $X = (1+o(1))\mathbb{E}[X] = (1+o(1))n\lambda(A(n))$.
	\end{proof}
	
	In Section \ref{sec:Extensions}, we will need a slight strengthening of the result above; namely, if we sample $n$ points uniformly from $[0,1]$, the lemma below guarantees that no relatively large subset of these points will occupy an interval much smaller than expected. 
	
	\begin{lemma}\label{lem:IntervalOccupied}
		Let $\delta=\omega\big(\frac{1}{\sqrt{n}}\big)$. Among $n$ points uniformly distributed on the interval $[0,1]$, with probability $1-o(1)$, every set of $\delta n$ points will occupy an interval of length at least $\frac{\delta}{2}(1-o(1))$.
	\end{lemma}
	
	\begin{proof}
		We begin by dividing $[0,1]$ into consecutive intervals of length $\delta$. By Lemma~\ref{lem:IntervalPointConcent}, with probability $1-o(1)$, there will be at most $(\delta+o(1))n$ vertices in any fixed one of these intervals, as $\delta = \omega\left(\frac{1}{n}\right)$. For $\delta=\omega\big(\frac{1}{\sqrt{n}}\big)$, there will in fact be at most $(\delta+o(1))n$ vertices in \textit{each}; as shown in Lemma~\ref{lem:IntervalPointConcent}, if $X$ is the number of vertices in a given interval of length $\delta$, then for any $\eps>0$, we have
		\[
		\Pr\left[X\geq (1+\eps)\delta n\right]\leq \frac{1}{\eps^2(\delta n)}.
		\]
		Then, taking a union bound over the $\frac{1}{\delta}$ consecutive length-$\delta$ intervals, with probability $1-\frac{1}{\delta}\cdot \frac{1}{\eps^2(\delta n)} = 1- \frac{1}{\eps^2\delta^2 n}$, each of these intervals contains at most $(1+\eps)\delta n$ vertices. So for any $\delta = \omega\big(\frac{1}{\sqrt{n}}\big)$, we can choose an appropriate $\eps = o(1)$ to conclude that with probability $1-o(1)$, each of the $\frac{1}{\delta}$ consecutive intervals contain at most $(1+o(1))\delta n$ vertices.
		
		Notice that any other interval of length $\delta$ in $[0,1]$ is contained entirely in at most two of these consecutive intervals. So with probability $1-o(1)$, any interval of length $\delta$ in $[0,1]$ will contain at most $(1+o(1))2\delta n$ vertices. Equivalently, and after a slight change of variables, with probability $1-o(1)$, every $\delta n$ vertices will occupy an interval of length at least $\frac{\delta}{2}(1-o(1))$.
	\end{proof}

	Now, we would like to be able to say something about the cliques in $W$-random graphs that we obtain by sampling points from smaller sets, for example, from neighborhoods around points at which $W= 1$. As in \cite{DHM17}, we define a \emph{subgraphon} of any graphon $W$ to be the restriction obtained by ``zooming in" on a subset of the sample space:
	
	\begin{definition}
		Given a graphon $W\colon[0,1]^2\to [0,1]$ and a	subset $A\subseteq [0,1]$ of positive measure, define the subgraphon $W|_{A\times A}\colon A^2\to [0,1]$ as the restriction of $W$ to $A\times A$, where we sample uniformly from the set $A$ to obtain a probability distribution on $A$.
	\end{definition}
	
	(Note that $W|_{A\times A}$ as defined above satisfies the definition of a graphon on a more general probability space $\Omega$.) Intuitively, if we break a graphon $W$ into subgraphons, its clique number will be at least the maximum clique number among the subgraphons and at most the sum of all their clique numbers. This intuition is formalized in the following lemma.
	
	\begin{lemma}\label{lem:SubCliqueN}
		Let $W\colon[0,1]^2\to [0,1]$ be a graphon, let $k\in\mathbb{N}$ be constant, and let $A_1,\dots, A_k\subseteq [0,1]$ be measurable sets depending on $n$ that partition $[0,1]$, where each $A_i=A_i(n)$ has measure $\lambda(A_i)= \omega\left(\frac{1}{n}\right)$. Then for each $i$, there exist $n_i^+,n_i^- = (1+o(1))\lambda(A_i)$, with $n_i^-\leq n_i^+$,	such that a.a.s., 
		\begin{enumerate}[(i)]
			\item\label{lem:SubCliqueUpperN}
			$
			\omega(\mathbb{G}(n,W))\leq    (1+o(1))\big[\omega(\mathbb{G}(n_1^+,W|_{A_1\times A_1})) + \cdots + \omega(\mathbb{G}(n_k^+,W|_{A_k\times A_k}))\big] ,
			$ and
			\item\label{lem:SubCliqueLowerN}
			$
			\omega(\mathbb{G}(n,W))\geq     (1+o(1))\cdot \omega(\mathbb{G}(n_i^-,W|_{A_i\times A_i}))
			$ for each $i\in\{1,\dots, k\}$.
		\end{enumerate}
	\end{lemma}
	
	\begin{proof}
		We will show that (\ref{lem:SubCliqueUpperN}) and (\ref{lem:SubCliqueLowerN}) each hold for a specific coupling of $\mathbb{G}(n,W)$ with $\left(\mathbb{G}(n_1^-,W|_{A_1\times A_1}),\dots,\mathbb{G}(n_k^-,W|_{A_k\times A_k})\right)$. But before doing so, we briefly argue that this suffices to prove the lemma for any coupling. By Theorem~\ref{thm:DolezalConcentration} (proved in \cite{DHM17}), we know that the clique number for any graphon is highly concentrated: for any graphon $U$, with probability $1-o(1)$, we have $\omega(\mathbb{G}(n,U)) = (1+o(1)) \mathbb{E}(\omega(\mathbb{G}(n,U)))$. So if  (\ref{lem:SubCliqueUpperN}) or (\ref{lem:SubCliqueLowerN}) holds for any specific choice of coupling, then it will hold for all, since each term in (\ref{lem:SubCliqueUpperN}) and (\ref{lem:SubCliqueLowerN}) changes by at most a factor of $1+o(1)$ regardless of the choice of coupling. 
		
		With this in mind, we prove (\ref{lem:SubCliqueLowerN}). By Lemma~\ref{lem:IntervalPointConcent}, when sampling vertices of the $W$-random graph $\mathbb{G}(n,W)$, the number of vertices in the set $A_i$ for each $i$ will be at least $n_i^- = (1-o(1))\lambda(A_i)n$ a.a.s., for an appropriate $o(1)$ function. So there is a coupling of $\mathbb{G}(n,W)$ with $\left(\mathbb{G}(n_1^-,W|_{A_1\times A_1}),\dots,\mathbb{G}(n_k^-,W|_{A_k\times A_k})\right)$ such that a.a.s.\ each $\mathbb{G}(n_1^-,W|_{A_1\times A_1})$ is contained as a subgraph in $\mathbb{G}(n,W)$.     For this coupling, (\ref{lem:SubCliqueLowerN}) automatically holds. 
		
		Explicitly, the coupling is constructed as follows: for each $n$, we sample the $n$ vertices of $\mathbb{G}(n,W)$. With probability $1-o(1)$, there will be at least $n_i$ vertices sampled from each $A_i$ (note: since $k$ is constant, taking a union bound over all the $A_i$ does not change this). Assume we are in this case (else, generate the other graphs independently). To generate each $\mathbb{G}\left(n_i^-,W|_{A_i\times A_i}\right)$, since at least $n_i^-$ of the vertices of $\mathbb{G}(n,W)$ are in $A_i$, then uniformly sample exactly $n_i^-$ of them. The subgraph induced on these vertices has distribution $\mathbb{G}\left(n_i^-,W|_{A_i\times A_i}\right)$.  We place no additional edges. Note that we re-sample all the vertices for each $n$ to generate this coupling, as opposed to adding on a vertex to $\mathbb{G}(n-1,W)$ to generate $\mathbb{G}(n,W)$. For this coupling, (\ref{lem:SubCliqueLowerN}) holds; thus as argued above, (\ref{lem:SubCliqueLowerN}) holds in general.
		
		Now we show that (\ref{lem:SubCliqueUpperN}) also holds for a similar coupling. By Lemma~\ref{lem:IntervalPointConcent}, in $\mathbb{G}(n,W)$, the number of vertices in the set $A_i$ will be at most $n_i^+ = (1+o(1))\lambda(A_i)n$ for an appropriate $o(1)$ function. We couple $\mathbb{G}(n,W)$ with $\left(\mathbb{G}(n_1^+,W|_{A_1\times A_1}),\dots,\mathbb{G}(n_k^+,W|_{A_k\times A_k})\right)$ as follows: for each $n$, we sample $n$ vertices for $\mathbb{G}(n,W)$. With probability $1-o(1)$, there will be at most $n_i^+$ of them in each $A_i$. If this happens, then to generate each $\mathbb{G}\left(n_i^+,W|_{A_i\times A_i}\right)$, take these vertices, and add in enough extra vertices (uniformly sampled from $A_i$) to make exactly $n_i^+$ total vertices in $A_i$. On these $n_i^+$ vertices, place all edges belonging to the copy of $\mathbb{G}(n,W)$ that we have sampled, and add edges between a new vertex $v$ and any other vertex $w$ with probability $W(v,w)$. Now, the subgraph induced on these vertices has distribution $\mathbb{G}\left(n_i^+,W|_{A_i\times A_i}\right)$.
		And we see that the max clique of $\mathbb{G}(n,W)$ is, at very most, the union of the max cliques of the graphs $\mathbb{G}\left(n_i^+,W|_{A_i\times A_i}\right)$. So for this coupling, (\ref{lem:SubCliqueUpperN}) holds, and as a consequence, holds for any coupling.
	
	\end{proof}

	Note that in the previous lemma, the quantities $n_i^-$ and $n_i^+$ are functions only of $n$ and $\lambda(A_i)$, and not of the graphon $W$; we will use this fact in the proof of Lemma~\ref{lem:UpperDini}.
	
	We finish this section with two lemmas showing that the clique number of $\mathbb{G}(n,W)$ is determined up to lower-order terms (in the case where this clique number is $\omega(\log n)$) by the local behavior of $W$ near points $(a,a)$ where $W(a,a) = 1$. In particular, if $W$ is bounded above by $U$ locally near points where $W(a,a)=1$, the following lemma tells us that the clique numbers of $W$ and $U$ will satisfy the same inequality, up to lower-order terms. Note that in the proof of the lemma below, there is nothing special about using graphons $W$ and $U$ on $[0,1]^2$; in particular, we can take graphons on $A^2$ for any positive-measure subset $A\subseteq [0,1]$, as long as both graphons have the \emph{same} domain. We use this fact in the proof of Lemma~\ref{lem:UpperDini}.
	
	\begin{lemma}\label{lem:LocalDominance}
		Let $W,U:[0,1]\to[0,1]^2$ be graphons equal to 1 at some point $(a,a)$, and essentially bounded away from 1 in some neighborhood of $(x,x)$ for all other $x\in[0,1]^2$. If there exists some neighborhood $N$ of $(a,a)$ on which $W(x,y)\leq U(x,y)$, then a.a.s.,
		\[\omega(\mathbb{G}(n,W))\leq (1+o(1))\cdot \omega(\mathbb{G}(n,U))+O(\log n).\]
	\end{lemma}
	
	\begin{proof}
		As in Lemma~\ref{lem:SubCliqueN}, we will show that the result holds for a specific coupling of $\mathbb{G}(n,W)$ and $\mathbb{G}(n,U)$, and as argued in the proof of Lemma~\ref{lem:SubCliqueN}, this will in fact suffice to prove it for any choice of coupling. We define our coupling as follows: first, sample the \emph{same} numbers $x_1,\dots, x_n$ for the vertices of both $\mathbb{G}(n,W)$ and $\mathbb{G}(n,U)$, and couple their edges in such a way that every pair of vertices $(i,j)$ with $(x_i,x_j)\in N$ is connected by an edge in $\mathbb{G}(n,W)$ only if it is also connected in $\mathbb{G}(n,U)$; this is possible because $W(x,y)\leq U(x,y)$ on the neighborhood $N$. The coupling of the remaining edges can be defined in any way (for concreteness, we may sample them independently for the two graphs). 
		 
		 Before proceeding further, we will restrict our view to a subset $I^2\subseteq N$, for some open interval $I$ containing $a$. Then we may define $n_{\text{in}}$ to be the (random) number of vertices $x_1,\dots x_n$ that are in the interval $I$, and consider the random graphs $\mathbb{G}(n_{\text{in}},W|_{I^2})$ and $\mathbb{G}(n_{\text{in}},U|_{I^2})$. We may generate them by taking the subgraphs of $\mathbb{G}(n,W)$ and $\mathbb{G}(n,U)$ respectively induced on these vertices (still using the coupling of $\mathbb{G}(n,W)$ and $\mathbb{G}(n,U)$ generated above). With this coupling, $\mathbb{G}(n_{\text{in}},W|_{I^2})$ is contained as a subgraph in $\mathbb{G}(n_{\text{in}},U|_{I^2})$; thus we may write
		\begin{equation}\label{eqn:BigCliqueIneq}
				\omega(\mathbb{G}(n_{\text{in}},W|_{I^2}))\leq\omega(\mathbb{G}(n_{\text{in}},U|_{I^2})).
		\end{equation}
		 And as $\mathbb{G}(n_{\text{in}},U|_{I^2})$ is a subgraph of $\mathbb{G}(n,U)$ in the coupling we have just defined, we may also write $\mathbb{G}(n_{\text{in}},U|_{I^2})\leq\mathbb{G}(n,U)$, or combining this with (\ref{eqn:BigCliqueIneq}),
		\begin{equation}\label{eqn:AllButLogClique}
				\omega(\mathbb{G}(n_{\text{in}},W|_{I^2}))\leq\omega(\mathbb{G}(n,U).
		\end{equation}		 
		 This is the essence of our proof; however, we still need to deal with all the vertices in $\mathbb{G}(n,W)$ that do not fall into the interval $I$, and ensure that they will not change the clique number of $\mathbb{G}(n,W)$ by too much. 
		 
		We deal with the remaining vertices as follows: let $n_\text{out}$ be the number of vertices not in $I$ (i.e., $n_{\text{out}}= n- n_\text{in}$). By the same reasoning just used, we may generate $\mathbb{G}(n_{\text{out}},W|_{([0,1]\setminus I)^2})$ as the subgraph of $\mathbb{G}(n,W)$ induced on the vertices counted by $n_\text{out}$.
		And given any partition of the vertices of a graph $G$, the clique number of $G$ is at most the sum of the clique numbers of the subgraphs induced on the parts of the partition. Here, given the partition of $[n]$ into $I$ and $[0,1]\setminus I$, this gives
		\begin{equation}\label{eqn:WholeCliqueIneq}
		 	\omega(\mathbb{G}(n,W))\leq \omega(\mathbb{G}(n_{\text{in}},W|_{I^2}))+\omega(\mathbb{G}(n_{\text{out}},W|_{([0,1]\setminus I)^2})).
		 \end{equation}
		 Combining this with (\ref{eqn:AllButLogClique}), we see that
		 \begin{equation}\label{eqn:almostThereClique}
		 	\omega(\mathbb{G}(n,W))\leq \omega(\mathbb{G}(n,U))+\omega(\mathbb{G}(n_{\text{out}},W|_{([0,1]\setminus I)^2})).
		 \end{equation}
		 Now, since the subgraphon $W|_{([0,1]\setminus I)^2}$ is essentially bounded away from 1 in some neighborhood of each $(x,x)$, we may apply Lemma~\ref{lem:TheImportanceOfBeingDiagonal} (proved below) to conclude that $\omega(\mathbb{G}(n_{\text{out}},W|_{([0,1]\setminus I)^2}))= O(\log(n_\text{out})) = O(\log n)$ a.a.s. (Note that we are taking some liberties in the application of Lemma~\ref{lem:TheImportanceOfBeingDiagonal}, in applying it to a graphon not defined on $[0,1]^2$, and in taking a random number of vertices; this can be justified formally, and does not significantly change the proof.) With this, (\ref{eqn:almostThereClique}) becomes
		 \[
		 			\omega(\mathbb{G}(n,W))\leq \omega(\mathbb{G}(n,U))+O(\log n).
		 \]
		 To finish, note that as shown in the proof of Lemma~\ref{lem:SubCliqueN}, for \emph{any} choice of coupling of $\mathbb{G}(n,W)$ and $\mathbb{G}(n,U)$, the clique numbers of each of these graphs will change by a factor of at most $1+o(1)$ a.a.s. Therefore, regardless of the choice of coupling,
		 \[
		 		\omega(\mathbb{G}(n,W))\leq (1+o(1))\omega(\mathbb{G}(n,U))+O(\log n)
		 \]
		 with probability $1-o(1)$, as desired.
	\end{proof}

	We end this section with a proof of Lemma~\ref{lem:TheImportanceOfBeingDiagonal}, restated here for the convenience of the reader.
	
	\DiagonalLemma*
	
	\begin{proof}[Proof of Lemma~\ref{lem:TheImportanceOfBeingDiagonal}]
		Associate to each point $(x,x)$ an open neighborhood $N(x,x)$ in $[0,1]^2$ on which the essential supremum of $W$ is $c(x) < 1$. These neighborhoods form an open cover of the (closed) diagonal line segment $D = \{(x,x): x\in[0,1]\}$. Because this set is compact, we may find an finite subcover of $D$ by neighborhoods $N(x,x)$. Taking the maximum essential supremum $c=c(x)$ of $W$ on any of these neighborhoods, we see that for some $\eps>0$, the essential supremum of $W$ is $c$ on the region $\{(x,y): |x-y|\leq \eps\}$.
		
		Now consider any $k$ vertices from $\mathbb{G}(n,W)$, and view them as points in $[0,1]$. By the pigeonhole principle, dividing $[0,1]$ into $1/\eps$ disjoint intervals of length $\eps$, of the $k$ points, there must be at least $\eps k$ points in some interval of length $\eps$. The probability that this subset forms a clique is at most $c^{\binom{\eps k}{2}}= c^{\Theta(k^2)}$, which also gives an upper bound on the probability that the original $k$ vertices formed a clique. So, taking a union bound, the probability that there exists any $k$-clique is at most
		\[
		\binom{n}{k}c^{\Theta(k^2)} \leq  \left(\frac{en}{(\frac{1}{c})^{\Theta(k)}\cdot k}\right)^k
		=\left(\frac{n}{(\frac{1}{c})^{\Theta(k)}}\right)^k.
		\]
		The cutoff at which this approaches zero is $k=\Theta(\log n)$. So for any graphon $W$ bounded away from 1 in some neighborhood of each point $(x,x)$, the clique number of $\mathbb{G}(n,W)$ is a.a.s.\ $O(\log n)$. 
	\end{proof}
	
	Notice that there was nothing special about the choice of $[0,1]^2$ in this result; the only property we used of the interval $[0,1]$ was its compactness. So in particular, Lemma~\ref{lem:TheImportanceOfBeingDiagonal} holds for a graphon defined on $A^2$ for any closed interval $A$, a fact that we will use several times throughout this paper.
	
	\section{$W$-random graphs with clique number $\Theta(\sqrt n)$}\label{sec:Sqrt}
	
	In this section, we prove Lemma~\ref{lem:Lipschitz} and Theorem~\ref{thm:DirectionalDerivativeCliqueNum}, which characterize a variety of $W$-random graphs with clique number $\Theta(\sqrt{n})$ in terms of the local behavior of $W$ at points where it is equal to $1$. We begin by finding the clique number of a specific family of random graphs; this will in fact suffice to prove both Lemma~\ref{lem:Lipschitz} and a more general result, of which Theorem~\ref{thm:DirectionalDerivativeCliqueNum} is a special case.
	
	\subsection{A family of examples with clique number $\Theta(\sqrt n)$}
	
	\begin{lemma}\label{lem:SqrtFamily}
		For any $r>0$, define the graphon $W_r(x,y) = (1-x)^r(1-y)^r$. The clique number of $\mathbb{G}(n,W_r)$ is a.a.s.\ $\Theta(\sqrt{n})$.
	\end{lemma}
	
	To prove Lemma~\ref{lem:SqrtFamily}, we begin by finding an upper bound on $\omega(\mathbb{G}(n,W_r))$; we will use the first moment method.
	
	\begin{lemma}\label{lem:SqrtUpperBd}
		The clique number of $\mathbb{G}(n,W_r)$ is a.a.s.\ at most $(1+o(1))\left(\frac{e}{r}\right)^{1/2}\cdot \sqrt{n}$.
	\end{lemma}
	
	\begin{proof}
		Write $X_k$ for the number of cliques of size $k$ in $\mathbb{G}(n,W_r)$. By Markov's inequality, $\omega(\mathbb{G}(n,W_r))$ is a.a.s.\ bounded above by any $k$ for which $\mathbb{E}[X_k]$ is asymptotically $0$. And for any $k$, writing $d\vec{x}$ in place of $dx_1\cdots dx_k$, we have
		\begin{align*}
		\mathbb{E}[X_k] &= \binom{n}{k}\int_{[0,1]^{k}}\prod_{\ell\neq m\in[k]}W(x_\ell,x_m)\ d\vec{x}\\
		&= \binom{n}{k}\int_{[0,1]^{k}}\prod_{\ell\neq m\in[k]}(1-x_\ell)^r\cdot (1-x_m)^r\ d\vec{x}\\
		& = \binom{n}{k}\left(\int_0^1 (1-x)^{r(k-1)}dx\right)^k\\
		& = \binom{n}{k}\left(\frac{1}{r(k-1)+1}\right)^k.
		\end{align*}
		For any $k$ that is $\omega(1)$, we have $\frac{1}{r(k-1)+1}= (1+o(1))\frac{1}{rk}$. And for any $k$ that is $\omega(1)$ but sublinear, it can be shown from Stirling's formula that $\binom{n}{k} = \left(\frac{en}{k}(1-o(1))\right)^k$. Therefore the above expression becomes
		\begin{align*}
		\mathbb{E}[X_k] = \left(\frac{en}{k}(1-o(1))\right)^k\left((1+o(1))\frac{1}{rk}\right)^k
		&=\left(\frac{en}{rk^2}(1-o(1))\right)^k.
		\end{align*}
		So the cutoff at which $\mathbb{E}[X_k]$ goes from asymptotically 0 to asymptotically infinity is when $k\sim\left(\frac{e}{r}\right)^{1/2}\cdot \sqrt{n}$, which, by Markov's inequality, gives an upper bound on the clique number of $\mathbb{G}(n,W_r)$ that will hold with probability $1-o(1)$.
	\end{proof}
	Ideally, we would like to prove a matching lower bound. However, such a bound may be difficult to establish, or even untrue, as the variance of the number of cliques in $\mathbb{G}(n,W_r)$ of any size of order $\Theta(\sqrt{n})$ is quite large (Lemma~\ref{cor:BigVarSqrtAndPoly} (\ref{cor:BigVarSqrt}) in the appendix). In particular, this means we cannot use the second moment method directly to prove a lower bound on the clique number $\omega(\mathbb{G}(n,W_r))$. (This argument is fleshed out more fully in the appendix.) These difficulties notwithstanding, we can at least prove a lower bound that matches up to a constant.
	
	\begin{lemma}\label{lem:SqrtLowerBd}
		The clique number of $\mathbb{G}(n,W_r)$
		is a.a.s.\ at least $ \left(\frac{1}{12er}\right)^{1/2}\cdot \sqrt{n}$.
	\end{lemma}
	
	\begin{proof} Our strategy is to directly compute a lower bound on the expected clique number for the graphon $W_r(x,y) = (1-x)^r(1-y)^r$ by guessing which vertices are most likely to form a large clique and showing that this does indeed happen with high probability. Suppose that for some constants $s$ and $t$, there are $sn^{1/2}$ vertices less than $tn^{-1/2}$ in $\mathbb{G}(n,W_r)$. (Note: the expected number of vertices less than $tn^{-1/2}$ is $tn^{1/2}$.) By Lemma~\ref{lem:IntervalPointConcent}, this will happen a.a.s.\ for some $t = (1+o(1))s$. We will show, for an appropriate choice of $s$, that if we do have such vertices, then a.a.s., the subgraph they induce will contain all but $k$ possible edges (for some appropriate choice of $k$ dependent on $n$). In total, this will show that the clique number is a.a.s.\ at least $s\sqrt{n}-k$, obtained by greedily deleting one vertex from each of the (up to) $k$ missing edges. 
		
	Concretely, for any constants $s$ and $t$, suppose that $\mathbb{G}(n,W_r)$ has $sn^{1/2}$ vertices less than $tn^{-1/2}$. The probability that any fixed set of $k$ potential edges is missing from the subgraph of $\mathbb{G}(n,W_r)$ induced on those vertices is at most
		\begin{align*}
		\prod_{k \text{ edges}}\left[1-(1-tn^{-1/2})^r(1-tn^{-1/2})^r\right]
		&= \left[1-(1-tn^{-1/2})^{2r}\right]^k\\
		&= \left[1-\left(1-2rtn^{-1/2}\cdot(1+o(1))\right)\right]^k,
		\end{align*}
		where the last equality follows by taking a binomial series expansion. 
		Simplifying this expression slightly, the probability that any fixed set of $k$ edges is missing is at most
		\begin{align*}
		\left[2rtn^{-1/2}\cdot(1+o(1))\right]^k .
		\end{align*}
		Now to bound the probability that there are $k$ or more edges missing, we take a union bound over all sets of $k$ possible edges in the subgraph induced on the $s\sqrt{n}$ vertices under consideration. The number of such sets is
		\begin{align*}
		\binom{\binom{s\sqrt{n}}{2}}{k}\leq  \binom{s^2n/2}{k} \leq \left(\frac{es^2n/2}{k}\right)^k.
		\end{align*}
		So in total, the probability to have $k$ or more missing edges is at most
		\begin{align*}
		\left(\frac{es^2n/2}{k}\right)^k \left(2rtn^{-1/2}\cdot(1+o(1))\right)^k
		& =  \left(\frac{erts^2\sqrt{n}}{k}\cdot (1+o(1))\right)^k.
		\end{align*}
		
		As argued above, for any constant $s$, with probability $1-o(1)$, there is a set of $s\sqrt{n}$ vertices less than $tn^{-1/2}$, for some $t = (1+o(1))s$. Given such a set, as just shown, the probability that the induced subgraph on these vertices is missing $k$ or more edges is at most $\left(\frac{erts^2\sqrt{n}}{k}\cdot (1+o(1))\right)^k$. Thus if this quantity is $o(1)$, then a.a.s.\ there is a clique of size at least $s\sqrt{n} - k$ in $\mathbb{G}(n,W_r)$, obtained by deleting one vertex from each missing edge. If we choose $k$ to be, for example, $\frac{1}{2}s\sqrt{n}$, then
		\[
		\left(\frac{erts^2\sqrt{n}}{k}\cdot (1+o(1))\right)^k = \left({2ers^2}\cdot (1+o(1))\right)^{\frac{1}{2}s\sqrt{n}}.
		\]
		This will be $o(1)$ as long as ${2ers^2} = 1-\Omega(1)$, or equivalently, $s^2= \frac{1-\Omega(1)}{2er}$. Taking any constant $s<\frac{1}{\sqrt{2er}}$ suffices, for instance $s = \frac{1}{\sqrt{3er}}$. Therefore, asymptotically almost surely, there will exist a clique of size at least $s\sqrt{n} - k  = \frac{1}{2}\cdot \frac{1}{\sqrt{3er}}\sqrt{n} = \left(\frac{1}{12er}\right)^{1/2}\cdot \sqrt{n}$. 
	\end{proof}
	
	Note that the bound in Lemma~\ref{lem:PolyLowerBd} can be tightened by optimizing the choice of $k$ in the proof above, but not to the point of matching the upper bound given in Lemma~\ref{lem:SqrtUpperBd}. Together, the upper and lower bounds in Lemmas \ref{lem:SqrtUpperBd} and \ref{lem:SqrtLowerBd} imply the $\Theta(\sqrt n)$ bound given in Lemma~\ref{lem:SqrtFamily} for $\omega(\mathbb{G}(n,W_r))$.
	
	\subsection{More general $W$-random graphs with clique number $\Theta(\sqrt n)$}\mbox{}\\

	We are now nearly ready to prove the main results of this section, Lemma~\ref{lem:Lipschitz} and Theorem~\ref{thm:DirectionalDerivativeCliqueNum}. We will reframe both results in a slightly broader setting and prove a more general version of Theorem~\ref{thm:DirectionalDerivativeCliqueNum}. This theorem characterizes the clique number of a $W$-random graph when $W$ has moderate directional derivatives at the points where it is equal to 1. However, even if the directional derivatives of a graphon $W$ do not exist at a given point, we can still have some notion of ``bounded derivatives" by looking at the limsup and the liminf of the difference quotient that defines the derivative.
	
	\begin{definition}\label{def:DiniDer} For a function $W\colon\mathbb{R}^k\to \mathbb{R}$, a point $x\in\mathbb{R}^k$, and a unit direction vector $d\in\mathbb{R}^k$, the \emph{upper Dini derivative} of $W$ at $x$ in direction $d$ is defined as
		\[
		W_+'(x,{d}) = \limsup_{h\to 0^+}\frac{W(x+h{d})-W({x})}{h}.
		\]
		The \emph{lower Dini derivative} of $W$ at $x$ in direction $d$ is
		\[
		W_-'(x,{d}) = \liminf_{h\to 0^+}\frac{W(x+h{d})-W({x})}{h}.
		\]
	\end{definition}
	
	Notice that if any directional derivative of a graphon $W$ exists, then it is equal to both the corresponding upper and lower Dini derivatives. Also, we have defined Dini derivatives only in directions corresponding to unit vectors; this is not necessary, but it makes several of the results and their proofs below slightly neater. We will use these definitions throughout the remainder of this section. We now show that a bound on the lower Dini derivatives of a graphon $W$ at a point $(a,a)$ with $W(a,a)=1$ provides a lower bound on the clique number of $\omega(\mathbb{G}(n,W))$.
	
	\begin{lemma}\label{lem:LowerDini}
		Let $W\colon[0,1]^2\to [0,1]$ be a graphon equal to 1 at some point $(a,a)$.
		\begin{enumerate}[(i)]
			\item If all lower Dini derivatives of $W$ at $(a,a)$ are bounded below by $-c$ for some constant $c\geq0$, then $\omega(\mathbb{G}(n,W)) = \Omega(\sqrt{n})$ a.a.s.
			\label{lem:LowerDiniNonzero}
			\item If all directional derivatives of $W$ at $(a,a)$ exist and are equal to zero, then $\omega(\mathbb{G}(n,W)) = \omega(\sqrt{n})$ a.a.s.\label{lem:LowerDiniZero}
		\end{enumerate}
	\end{lemma}
	
	Before proving Lemma~\ref{lem:LowerDini}, we quickly show how it implies Lemma~\ref{lem:Lipschitz}, which is restated here for the convenience of the reader.
	
	\firstLipschitz*
	
	\begin{proof}[Proof of Lemma~\ref{lem:Lipschitz}]
		We begin by recalling a definition: $W$ is locally Lipschitz at the point $(a,a)$ if there exists a neighborhood $U$ of $(a,a)$ and a constant $c\geq 0$ such that for all points $(x,y)$ in $U$,
		\[
		|W(x,y)-W(a,a)|\leq c\cdot\|(x,y)-(a,a)\|.
		\]
		where in the line above, $\|\cdot\|$ represents the $\ell^2$ norm. (Note that any other norm would produce an equivalent definition, as all norms on $\mathbb{R}^2$ are equivalent up to a constant.) If $W$ satisfies this condition, and if $W(a,a) = 1$, then since $W(x,y)\leq 1$, the inequality above becomes
		\[
		W(x,y)-W(a,a)\geq -c\cdot \|(x,y)-(a,a)\|.
		\]
		Now for any $(x,y)\in U$, write $(x,y)-(a,a) = hd$ for a unit direction vector $d$; with this substitution, the inequality above is equivalent to
		\begin{equation}\label{eqn:DiniDiffQuotient}
				\frac{W((a,a)+hd)-W(a,a)}{h}\geq -c.
		\end{equation}
		Indeed, for any unit direction vector $d$, and for $h$ sufficiently small, the point $(a,a)+hd$ will be in the neighborhood $U$, and this inequality will hold. Thus, taking a liminf of inequality (\ref{eqn:DiniDiffQuotient}) for each $d$ as $h\to 0^+$, we see that by definition, all the lower Dini derivatives of $W$ at $(a,a)$ are at least $-c$. Then Lemma~\ref{lem:Lipschitz} follows immediately from Lemma~\ref{lem:LowerDini} (\ref{lem:LowerDiniNonzero}).
	\end{proof}
	We now give the proof of Lemma~\ref{lem:LowerDini}.
	
	\begin{proof}[Proof of Lemma~\ref{lem:LowerDini}]		
		We begin with part (\ref{lem:LowerDiniNonzero}). Roughly, our proof strategy will be to locally bound $W$ from below by a graphon in the family $\{W_r\}_{r\in\mathbb{R}^+}$, thereby bounding the clique number of $\mathbb{G}(n,W)$ from below by $\omega(\mathbb{G}(n,W_r)) = \Theta(\sqrt{n})$.
		
		Take any constant $\eps>0$, and notice that the graphon
		\[
		W_r(x,y) = W_{\sqrt{2}(c+\eps)}(x,y) = (1-x)^{\sqrt{2}(c+\eps)}(1-y)^{\sqrt{2}(c+\eps)}
		\]
		has directional derivatives at most $-\frac{r}{\sqrt{2}}=-(c+\eps)$ at $(0,0)$, achieved in the direction $\big(\frac{1}{\sqrt{2}},\frac{1}{\sqrt{2}}\big)$. 
		Also, $W_r(0,0)=1$. Then since $W(a,a)=1$, and the lower Dini derivatives of $W$ are at least $-c$, we see that up to translation of the domain, $W$ is bounded below by $W_r$ in some neighborhood of $(a,a)$. Therefore, by Lemma~\ref{lem:LocalDominance}, we have 
   \begin{equation}\label{eqn:FirstDiniLowerNonzero}
				\omega(\mathbb{G}(n,W)) \geq 	(1-o(1))\omega(\mathbb{G}(n,W_r)) - O(\log n)
	\end{equation}
		a.a.s. Note that we have not assumed that $W$ is bounded away from 1 away from the point $(a,a)$; however, we may still apply Lemma~\ref{lem:LocalDominance}, as we are only looking for a lower bound on $\omega(\mathbb{G}(n,W))$. 
		
		And by Lemma~\ref{lem:SqrtLowerBd}, the clique number of $\mathbb{G}(n,W_r)$ is at least $ \left(\frac{1}{12er}\right)^{1/2}\cdot \sqrt{n}$; thus (\ref{eqn:FirstDiniLowerNonzero}) becomes
		\begin{align*}
		\omega(\mathbb{G}(n,W))
		&\textstyle \geq (1-o(1))\left(\frac{1}{12er}\right)^{1/2}\cdot \sqrt{n} - O(\log n) = \Theta(\sqrt{n}).
		\end{align*}
		This proves part  (\ref{lem:LowerDiniNonzero}).
		
		The proof of (\ref{lem:LowerDiniZero}) is similar; if all directional derivatives of $W$ are equal to zero, then for any constant $r>0$, consider the graphon $W_{r}(x,y) = (1-x)^{r}(1-y)^{r}$. Since the directional derivatives of $W_r$ at $(0,0)$ are at most $-\frac{r}{\sqrt{2}}$, we have $W\geq W_{r}$ in some neighborhood of $(a,a)$, up to translation of the domain. Thus, again by Lemma~\ref{lem:LocalDominance}, 
		\[\omega(\mathbb{G}(n,W))\geq (1-o(1))\omega(\mathbb{G}(n,W_{r})-O(\log n) 
		\]
		And as above, this yields
		\begin{align*}
		\omega(\mathbb{G}(n,W))
		&\geq (1-o(1))\cdot\left(\tfrac{1}{12er}\right)^{1/2}-O(\log n)\\
		& = (1-o(1))\left(\tfrac{1}{12er}\right)^{1/2}\cdot \sqrt{n}.
		\end{align*}
		Then since we can choose $r$ arbitrarily small, we see that
		\[
		\omega(\mathbb{G}(n,W)) = \omega(\sqrt{n}),
		\]
		asymptotically almost surely, completing the proof of part (\ref{lem:LowerDiniZero}).
	\end{proof}
	
	Just as the a bound on the lower Dini derivatives of a graphon $W$ gives us a lower bound on the clique number of a $W$-random graph, a bound on the upper Dini derivatives will give us an upper bound. Since we are proving an upper bound on clique number, we will add the assumption that the graphon under consideration is only equal to 1 at a finite number of points $(a,a)$. Together with Lemma~\ref{lem:Lipschitz}, the following result will prove Theorem~\ref{thm:DirectionalDerivativeCliqueNum}.
	
	\begin{lemma}\label{lem:UpperDini}
		Let $W\colon[0,1]^2\to [0,1]$ be a graphon equal to $1$ at some collection of points $(a_1,a_1),\dots, (a_k,a_k)$, and essentially bounded away from 1 in some neighborhood of each other point $(x,x)$ for $x\in[0,1]$. Then
		\begin{enumerate}[(i)]
			\item if all upper Dini derivatives of $W$ at $(a_1,a_1),\dots, (a_k,a_k)$ are uniformly bounded away from zero then $\omega(\mathbb{G}(n,W)) = O(\sqrt{n})$ a.a.s.,\ and \label{lem:UpperDiniNonzero}
			\item if all upper Dini derivatives of $W$ at $(a_1,a_1),\dots, (a_k,a_k)$ are equal to $-\infty$, then $\omega(\mathbb{G}(n,W)) = o(\sqrt{n})$ a.a.s.\label{lem:UpperDiniInfty}
		\end{enumerate}
	\end{lemma}
	Before giving the proof, we briefly show how Theorem~\ref{thm:DirectionalDerivativeCliqueNum} follows as a direct consequence of part (\ref{lem:UpperDiniNonzero}) of this lemma, together with Lemma~\ref{lem:Lipschitz}.
	
	\begin{proof}[Proof of Theorem~\ref{thm:DirectionalDerivativeCliqueNum}]
		If $W$ is equal to $1$ at the points $(a_1,a_1),\dots, (a_k,a_k)$, and its directional derivatives exist and are uniformly bounded away from $-\infty$ at these points, then as argued in the proof of Lemma~\ref{lem:Lipschitz}, $W$ is locally Lipschitz at these points. Therefore, we can apply Lemma~\ref{lem:Lipschitz} and conclude that $\omega(\mathbb{G}(n,W)) = \Omega(\sqrt n)$ asymptotically almost surely.
		
		Similarly, if the directional derivatives of $W$ are uniformly bounded away from $0$ at the points $(a_1,a_1),\dots, (a_k,a_k)$, and if $W$ is essentially bounded away from 1 in some neighborhood of each other point $(x,x)$ for $x\in[0,1]$, then we may apply Lemma~\ref{lem:UpperDini} (\ref{lem:UpperDiniNonzero}) to obtain $\omega(\mathbb{G}(n,W)) = O(\sqrt n)$ a.a.s.\ Therefore $\omega(\mathbb{G}(n,W)) = \Theta(\sqrt{n})$ asymptotically almost surely, as desired.
	\end{proof}
	
	We now prove the lemma above.
	
	\begin{proof}[Proof of Lemma~\ref{lem:UpperDini}]
		First, divide $[0,1]$ into subintervals so that each contains only one point of interest; specifically, divide $[0,1]$ at each point $a_i$ and at an arbitrary point between each pair $a_i$ and $a_{i+1}$. This will produce a partition of $[0,1]$ into a total of $2k$ subintervals $A_1,\dots, A_{2k}$ so that for each $A_i$, either the left or right endpoint is one of the values $a_j$, and $W|_{A_i\times A_i}$ is  essentially bounded away from 1 in some neighborhood of each other $(x,x)\neq (a_j,a_j)$. We will bound the clique number of $\mathbb{G}(n,W)$ in terms of the subgraphons $W|_{A_i\times A_i}$. By Lemma~\ref{lem:SubCliqueN} (\ref{lem:SubCliqueUpperN}), a.a.s.,
		\begin{equation}\label{eqn:DiniUpper}
		\omega(\mathbb{G}(n,W))\leq    (1+o(1))\big[\omega(\mathbb{G}(n_1^+,W|_{A_1\times A_1})) + \cdots + \omega(\mathbb{G}(n_{2k}^+,W|_{A_{2k}\times A_{2k}}))\big]
		\end{equation}
		where each $n_i^+$ is of the form $n_i^+ =(1+o(1))\lambda(A_i)n$, and is a function only of $n$ and $\lambda(A_i)$, and not of $W$ (this fact follows from the proof of Lemma~\ref{lem:SubCliqueN}). 
		
		Now, to prove part (\ref{lem:UpperDiniNonzero}), take any $A_i$, and suppose the upper Dini derivatives of $W|_{A_i\times A_i}$ are at most $-c$ at $(a,a)$, for the endpoint $a$ of $A_i$ at which $W(a,a) = 1$. For any $0<\eps<c$, consider the graphon \[
		W_r(x,y) = W_{{c-\eps}}(x,y) = (1-x)^{c-\eps}(1-y)^{c-\eps}.
		\] 
		We have $W_{r}(0,0) = 1$, and at the point $(0,0)$, all the directional derivatives of $W_{{c+\eps}}$ are at least $-r=-c+\eps$ (achieved in the directions $(0,1)$ and $(1,0)$). Therefore, up to translation and/or reflection, $W|_{A_i\times A_i}$ is bounded above by $ W_{r}$ in some neighborhood, and essentially bounded away from 1 near each $(x,x)$ outside that neighborhood. The same statement also holds if we replace $W_r$ with $W_r|_{[0,\ell]^2}$, for $\ell=\lambda(A_i)$, and in this case, up to translation and/or reflection of the domain, $W|_{A_i\times A_i}$ and $W_r|_{[0,\ell]^2}$ are graphons on the \emph{same} interval. Therefore we may apply Lemma~\ref{lem:LocalDominance} to conclude that
		\begin{equation}\label{eqn:UpperDiniSub1}
			\omega(\mathbb{G}(n_i^+,W|_{A_i\times A_i})) \leq (1+o(1))\cdot \omega(\mathbb{G}(n_1^+,W_r|_{[0,\ell]^2}))+O(\log n_i^+)
		\end{equation}
		a.a.s. And by Lemma~\ref{lem:SubCliqueN} (\ref{lem:SubCliqueLowerN}), 
		\begin{equation*}		 (1+o(1))\cdot\omega(\mathbb{G}(n_i^-,W_r|_{[0,\ell]^2}))\leq \omega(\mathbb{G}(n,W_r)) 
		\end{equation*}		
		a.a.s.\ as well, where $n_i^- = (1-o(1))\lambda(A_i)n=(1-o(1))n_i^+$. Equivalently, rearranging slightly,
		\begin{align}
		\notag\omega(\mathbb{G}(n_1^+,W_r|_{[0,\ell]^2}))&\leq (1-o(1))\cdot \omega(\mathbb{G}(n(1+o(1)),W_r)) \\
		&\leq (1-o(1))\sqrt{\tfrac{e}{r}}\cdot \sqrt{n(1+o(1))} \label{eqn:UpperDiniSubNote}\\
		&= (1-o(1))\sqrt{\tfrac{e}{r}}\cdot \sqrt{n},\label{eqn:UpperDiniSub2}
		\end{align}		
		where (\ref{eqn:UpperDiniSubNote}) is a direct application of Lemma~\ref{lem:SqrtUpperBd}. Together, (\ref{eqn:UpperDiniSub1}) and (\ref{eqn:UpperDiniSub2}) imply that 
		\[
			\omega(\mathbb{G}(n_i^+,W|_{A_i\times A_i})) \leq  (1-o(1))\sqrt{\tfrac{e}{r}}\cdot \sqrt{n} = O(\sqrt{n})
		\]
		a.a.s. Since this is true for each $i$, equation (\ref{eqn:DiniUpper}) becomes.
		\begin{align*}
		\omega(\mathbb{G}(n,W))\leq 2k\cdot O(\sqrt{n}) = O(\sqrt{n}),
		\end{align*}
		proving part (\ref{lem:UpperDiniNonzero}).
		\\
		\\To prove part (\ref{lem:UpperDiniInfty}), recall that the directional derivatives of $W_r$ at $(0,0)$ are at least $-r$. So if the upper Dini derivatives of $W$ are $-\infty$ at each of the points $(a_i,a_i)$, then for each $A_i$ and any $r>0$, we have (up to translation and/or reflection) $W|_{A_i\times A_i}\leq W_r$ locally on some neighborhood, and $W|_{A_i\times A_i}$ is essentially bounded away from 1 near each $(x,x)$ outside that neighborhood. Then for any $r>0$, as argued in the proof of part (\ref{lem:UpperDiniNonzero}), equations (\ref{eqn:UpperDiniSub1}) and (\ref{eqn:UpperDiniSub2}) hold here as well, again implying that
		\begin{align*}
		\omega(\mathbb{G}(n_i^+,W|_{A_i\times A_i})) \leq  (1-o(1))\sqrt{\tfrac{e}{r}}\cdot \sqrt{n}
		\end{align*}
		a.a.s. Then since we can choose $r$ arbitrarily large, we conclude that $\omega(\mathbb{G}(n_i^+,W|_{A_i\times A_i})) = o(\sqrt n)$ a.a.s. Substituting into (\ref{eqn:DiniUpper}), this gives
		\[
		\omega(\mathbb{G}(n,W))\leq 2k\cdot o(\sqrt n) = o(\sqrt n)
		\]
		asymptotically almost surely, completing the proof of (\ref{lem:UpperDiniInfty}).
	\end{proof}
	
	\section{A family of $W$-random graphs with clique number $\Theta(n^\alpha)$}\label{sec:Poly}
	
	As seen in the previous section, graphons $W$ with moderate local growth near points where $W(x,x) = 1$ produce $W$-random graphs with clique numbers $\Theta(\sqrt{n})$. In this section, we prove Theorem~\ref{thm:PolyFamCliqueNum}, which introduces a family of graphons with clique numbers $\Theta(n^\alpha)$ for any $\alpha\in(0,1)$. The members $W$ of this family corresponding to $\alpha\neq \frac{1}{2}$ have directional derivatives either $0$ or $-\infty$ at points where $W(x,x) = 1$, consistent with the results of the previous section. In this section, we also prove Theorem~\ref{thm:AlphaHolder}, which characterizes a larger class of $W$-random graphs with clique numbers $\Omega(n^\alpha)$, and Proposition~\ref{prop:TaylorSeriesAlmostLinear}, which gives an example of a $W$-random graph with clique number $n^{1-o(1)}$. We begin by proving Theorem~\ref{thm:PolyFamCliqueNum}, restated here for the convenience of the reader.
	
	\PolyFamClique*
	
	We will prove Theorem~\ref{thm:PolyFamCliqueNum} in very much in the same way as Theorem~\ref{lem:SqrtFamily}; first, we prove an upper bound on the clique number of $\mathbb{G}(n,U_r)$ by the first moment method.
	
	\begin{lemma}\label{lem:PolyUpperBd}
		For any $r>0$, the clique number of the random graph $\mathbb{G}(n,U_r)$ is at most $(1+o(1))\cdot  \left( \Gamma\left(1+\tfrac{1}{r}\right)e\right)^{\frac{r}{r+1}}\cdot n^{\frac{r}{r+1}} = \Theta(n^{\frac{r}{r+1}})$.
	\end{lemma}
	
	\begin{proof}
		For any $r>0$, write $X_k$ for the number of cliques of size $k$ in $\mathbb{G}(n,U_r)$. By Markov's inequality, the expected clique number of any random graph is a.a.s.\ bounded above by any value of $k$ for which $\mathbb{E}[X_k]=o(1)$. And for any $k$,
		\begin{align*}
		\mathbb{E}[X_k] & = \binom{n}{k}\int_{[0,1]^{k}}\prod_{\ell\neq m\in[k]}(1-x_\ell^r)\cdot (1-x_m^r)\ d\vec{x}\\
		& = \binom{n}{k}\left(\int_0^1 (1-x^r)^{k-1}\,dx\right)^k.
		\end{align*}
		Using the change of variables $u = x^r$, this expression becomes
		\begin{align*}
		& \binom{n}{k}\left(\frac{1}{r}\cdot \int_0^1 u^{\frac{1}{r}-1}(1-u)^{k-1}\,du\right)^k
		=\binom{n}{k}\left(\frac{1}{r}\cdot \frac{\Gamma(k)\Gamma(\frac{1}{r})}{\Gamma(k+\frac{1}{r})}\right)^k,
		\end{align*}
		where the last equality follows from the definition of the beta function, and its relationship to the gamma function. Simplifying slightly, we obtain
		\begin{align}\label{eqn:ExpPolyCliques}
		\mathbb{E}[X_k] & =\binom{n}{k} \left(\frac{\Gamma(k)\Gamma(1+\frac{1}{r})}{\Gamma(k+\frac{1}{r})}\right)^k.
		\end{align}
		And for any $k$ that is $\omega(1)$ but sublinear,  we have $\binom{n}{k} = \left(\frac{en}{k}(1-o(1))\right)^k$; thus (\ref{eqn:ExpPolyCliques}) becomes
		\begin{align*}
		\mathbb{E}[X_k] & =\left(\frac{en}{k}(1-o(1))\right)^k\left(\frac{\Gamma(k)\Gamma(1+\frac{1}{r})}{\Gamma(k+\frac{1}{r})}\right)^k.
		\end{align*}
		To obtain explicit asymptotics for this expression, we use Stirling's formula for the Gamma function, which states that for $x\to\infty$, 
		\[
		\Gamma(x)= (1+o(1)) \sqrt{(2\pi(x-1))}\left(\frac{x-1}{e}\right)^{x-1}.\]
		From this, for any fixed $r>0$ and $k\to\infty$, it follows that
		\begin{align}
		\frac{\Gamma(k)}{\Gamma(k+\frac{1}{r})} 
		&=(1+o(1))k^{-\frac{1}{r}}.\notag
		\end{align}
		Substituting this into (\ref{eqn:ExpPolyCliques}), we see that
		\begin{align*}
		\mathbb{E}[X_k]
		& =\left(\frac{en}{k}(1-o(1))\right)^k\left((1+o(1))k^{-\frac{1}{r}}\cdot \Gamma(1+\tfrac{1}{r})\right)^k\\
		& =\left(\frac{en\cdot \Gamma(1+\tfrac{1}{r})}{k^{1+\frac{1}{r}}}(1+o(1))\right)^k.
		\end{align*}
		Therefore the cutoff at which $\mathbb{E}[X_k]$ goes from asymptotically 0 to asymptotically infinity is when $k^{1+\frac{1}{r}}\sim en\cdot \Gamma(1+\tfrac{1}{r})$, or equivalently, $k\sim \left(\Gamma\left(1+\tfrac{1}{r}\right)e\right)^{\frac{r}{r+1}}\cdot n^{\frac{r}{r+1}}$. Hence, with probability $1-o(1)$, the clique number of $\mathbb{G}(n,U_r)$ is at most
		\begin{align*}
		k &=(1+o(1))\cdot (\Gamma\left(1+\tfrac{1}{r}\right)e)^{\frac{r}{r+1}}\cdot n^{\frac{r}{r+1}} = \Theta(n^{\frac{r}{r+1}}).\qedhere
		\end{align*}
	\end{proof}
	Now we will prove a lower bound on the clique number of $\mathbb{G}(n,U_r)$ - as in the previous section, it will match the upper bound up to a constant.    
	\begin{lemma}\label{lem:PolyLowerBd}
		The clique number of $\mathbb{G}(n,U_r)$ is a.a.s.\ at least $\frac{1}{2} \cdot e^{-\frac{2}{1+r}}\cdot n^{\frac{r}{r+1}}$.
	\end{lemma}
	
	\begin{proof}
		
		As in the proof of Lemma~\ref{lem:SqrtLowerBd}, we will directly compute a lower bound on the expected clique number for $\mathbb{G}(n,U_r)$ by guessing which vertices are most likely to form a large clique and showing that this does in fact happen with high probability. Suppose that there are $sn^{\frac{r}{r+1}}$ vertices less than $tn^{-\frac{1}{r+1}}$. (Note that the expected number of such vertices is $n\cdot tn^{-\frac{1}{r+1}} = tn^{\frac{r}{r+1}}$.) By Lemma~\ref{lem:IntervalPointConcent}, for any constant $s$, there is some $t = (1+o(1))s$, such that this will occur with probability $1-o(1)$. Then, given a set of $sn^{\frac{r}{r+1}}$ such vertices, what is the probability that the subgraph they induce is missing at most $k$ edges? The probability that any fixed set of $k$ potential edges is missing is at most
		\begin{align*}
		\prod_{k \text{ edges}}&\left[1-(1-(tn^{-\frac{1}{r+1}})^r)(1-(tn^{-\frac{1}{r+1}})^r)\right]\\
		&= \left[1-(1-(tn^{-\frac{1}{r+1}})^r)^2\right]^k\\
		&= \left[t^rn^{-\frac{r}{r+1}}\left(2-t^rn^{-\frac{r}{r+1}}\right)\right]^k\\
		&\leq \left(t^rn^{-\frac{r}{r+1}}\cdot 2\right)^k.
		\end{align*}
		Then, by a union bound, the probability that there exists any set of $k$ edges missing from the induced subgraph on these $sn^{\frac{r}{r+1}}$ vertices is
		\begin{align*}
		\binom{\binom{sn^{\frac{r}{r+1}}}{2}}{k} \cdot \left(t^rn^{-\frac{r}{r+1}}\cdot 2\right)^k
		&\leq \left(\frac{e(s^2n^{\frac{2r}{r+1}}/2)}{k}\right)^k\cdot \left(t^rn^{-\frac{r}{r+1}}\cdot 2\right)^k \\
		& =  \left(\frac{et^rs^2n^{\frac{r}{r+1}}}{k}\right)^k .
		\end{align*}
		If we choose $k$ to be, for example $\frac{1}{2}sn^{\frac{r}{r+1}}$, then this is equal to $\left({2et^rs}\right)^{\frac{1}{2}sn^{\frac{r}{r+1}}}$. As long as ${2et^rs} = 1-\Omega(1)$, or equivalently, $s^{1+r}= \frac{1-\Omega(1)}{2e}$, we will have $\left({2et^rs}\right)^{\frac{1}{2}sn^{\frac{r}{r+1}}} = o(1)$. Taking any constant $s<(2e)^{-\frac{1}{1+r}}$ suffices, for example $s = e^{-\frac{2}{1+r}}$. Therefore, for such a constant $s$, the induced subgraph on the $sn^{\frac{r}{r+1}}$ vertices under consideration is missing at most $k=\frac{1}{2}sn^{\frac{r}{r+1}}$ edges with probability $1-o(1)$. Deleting one vertex from each of these non-edges, we obtain a clique of size at least $sn^{\frac{r}{r+1}}-k = \frac{1}{2}sn^{\frac{r}{r+1}} =
		\frac{1}{2} \cdot e^{-\frac{2}{1+r}}\cdot n^{\frac{r}{r+1}}$ asymptotically almost surely.
	\end{proof}

	Notice that, as in Lemma~\ref{lem:SqrtLowerBd}, this does not quite match the upper bound of $ (\Gamma\left(1+\tfrac{1}{r}\right)e)^{\frac{r}{r+1}}\cdot n^{\frac{r}{r+1}}$; we could narrow the gap somewhat by optimizing parameters in the proof just given, but not to the point of closing it entirely. And as in Section \ref{sec:Sqrt}, the number of cliques of any size of the order $\Theta\big(n^{\frac{r}{r+1}}\big)$ in $\mathbb{G}(n,U_r)$ has quite high variance, which tells us that we cannot directly apply the second moment method to show that the lower bound we have given is tight (as indeed, it may not be). This argument is fleshed out more fully in the appendix, with a variance bound given by Lemma~\ref{cor:BigVarSqrtAndPoly} (\ref{cor:BigVarPoly}).
	
	We now use Theorem~\ref{thm:PolyFamCliqueNum} to prove Theorem~\ref{thm:AlphaHolder}, restated below. But first, we briefly discuss the continuity hypothesis in Theorem~\ref{thm:AlphaHolder}.
	
	\begin{definition}
		A graphon $W$ is \emph{locally $\alpha$-H\"older continuous} at $(a,a)$ if there exists some neighborhood $U$ of $(a,a)$ and some constant $C>0$ such that for all points \mbox{$(a,a)+(x,y)\in U$},
		\begin{equation}\label{eqn:AlphaHolderDef}
		\big|W\big((a,a)+(x,y)\big) - W(a,a)\big|<C\|(x,y)\|^\alpha,  
		\end{equation}
		where $\|\cdot\|$ may be taken to represent any fixed norm on $\mathbb{R}^2$.
	\end{definition}
	Typically, local $\alpha$-H\"older continuity is defined only for $\alpha\in[0,1]$; however, everything we do here will in fact hold and have meaning for larger $\alpha$ as well. On an interval, $\alpha$-H\"older continuity with $\alpha>1$ holds only for a constant function, but this is not the case for local $\alpha$-H\"older continuity at a single point, which may be achieved by a non-constant function whose derivatives are equal to zero at the point in question. 
	\AlphaHolder*
	
	\begin{proof}[Proof of Theorem~\ref{thm:AlphaHolder}]
		If $W$ is $\alpha$-H\"older continuous at $(a,a)$, then there exist $C>0$ and a neighborhood $U$ of $(a,a)$ such that (\ref{eqn:AlphaHolderDef}) is satisfied. For convenience, we will use the infinity norm. Then, since $W(a,a) = 1$, (\ref{eqn:AlphaHolderDef}) becomes
		\begin{equation}\label{eqn:AlphaHolder2}
		1-W\big((a,a)+(x,y)\big)<C\cdot \max(x,y)^\alpha.  
		\end{equation}
		With this in hand, we will prove a lower bound on the clique number of $W$ by bounding $W$ from below locally by a slightly modified member of the family $\{U_r\}$. Assume without loss of generality that the constant $C$ is at least 1. Then we define
		\[
		U_{\alpha,C}(x,y) =
		\begin{cases}
		(1-Cx^\alpha)(1-Cy^\alpha) &\text{ for $x,y\in[0,\frac{1}{C^{1/\alpha}}]$, and}\\
		0&\text{ otherwise.}
		\end{cases}\]
		Notice that for $x,y\in[0,\frac{1}{C^{1/\alpha}}]$, 
		\begin{align*}
		1-U_{\alpha,C}(x,y) &= 1-(1-Cx^\alpha)(1-Cy^\alpha)\\
		&=Cx^\alpha + Cy^\alpha(1-Cx^\alpha)\\
		&\geq Cx^\alpha
		\end{align*}
		Similarly, we have $1-U_{\alpha,C}(x,y) \geq Cy^\alpha$; thus $1-U_{\alpha,C}(x,y) \geq C\cdot \max{(x,y)}^\alpha$. 
		Therefore, by (\ref{eqn:AlphaHolder2}), we can write
		\begin{align*}
		1-W\big((a,a)+(x,y)\big)<C\cdot \max(x,y)^\alpha\leq 1-U_{\alpha,C}(x,y).
		\end{align*}
		for $(x,y)$ in some neighborhood of $(0,0)$. So up to translation, $W$ is bounded below by $U_{\alpha,C}$ in some neighborhood of $(a,a)$. (To be precise, we have only shown this in one quadrant, but this is sufficient for our purposes here.) And indeed, the clique number of $\mathbb{G}(n,U_{\alpha,C})$ is $\Theta(n^{\frac{\alpha}{\alpha+1}})$, as with $U_\alpha$.
		
		To see this, notice that in $\mathbb{G}(n,U_{\alpha,C})$, by Lemma~\ref{lem:IntervalPointConcent}, there will be $\sim \frac{1}{C^{1/\alpha}}n$ vertices selected from $[0,\frac{1}{C^{1/\alpha}}]$, and they will be uniform on this interval. If  $x$ and $y$ are uniform on $[0,\frac{1}{C^{1/\alpha}}]$, then $C^{1/\alpha}x, C^{1/\alpha}y$ are uniform on $[0,1]$. And for $x,y\in [0,\frac{1}{C^{1/\alpha}}]$, we have $U_{\alpha,C}(x,y) =    U_\alpha(C^{1/\alpha}x,C^{1/\alpha}y)$. So $\mathbb{G}(n,U_{\alpha,C})$ has the same distribution as a $U_\alpha$-random graph with approximately $\frac{1}{C^{1/\alpha}}n$ vertices, which has clique number $\Theta\big(( \frac{1}{C^{1/\alpha}}n)^{\frac{\alpha}{\alpha+1}}\big) = \Theta(n^{\frac{\alpha}{\alpha+1}})$ by Theorem~\ref{thm:PolyFamCliqueNum}.
		
		Now, given that $W$ is locally bounded below by $U_{\alpha,C}$ at $(a,a)$, and that $\mathbb{G}(n,U_{\alpha,C})$ has clique number $\Theta(n^{\frac{\alpha}{\alpha+1}})$, we may use the same argument as in Lemma~\ref{lem:LowerDini}; namely, we apply  Lemma~\ref{lem:LocalDominance}, which gives
		\begin{equation*}
		\omega(\mathbb{G}(n,W)) \geq 	(1-o(1))\cdot \omega(\mathbb{G}(n,U_{\alpha,C})) - O(\log n) = \Theta(n^{\frac{\alpha}{\alpha+1}})
		\end{equation*}
		a.a.s. Therefore $\omega(\mathbb{G}(n,W)=\Omega(n^{\frac{\alpha}{\alpha+1}})$ asymptotically almost surely.
	\end{proof}
	
	We end this section with a proof of Proposition~\ref{prop:TaylorSeriesAlmostLinear}, restated here.
	
	\TaylorSeriesAlmostLinear*

	\begin{proof}
		Our proof consists of two parts: first, for each $r\in\mathbb{N}$, we will show that $W$ is bounded below by $U_r$ locally in some neighborhood of $(0,0)$. We will then use Lemma~\ref{lem:LocalDominance} and the bound on $\omega(\mathbb{G}(n,U_r))$ given by Lemma~\ref{lem:PolyLowerBd} to give a lower bound on $\omega(\mathbb{G}(n,W))$. 
		
		We begin by looking at the (two-variable) Taylor polynomial of $W(x,y)$ about (0,0) of order $r$, for $r\in\mathbb{N}$. It is well known that $f(x)$, as defined above, is smooth on $\mathbb{R}$; this implies that $W$ is smooth on $\mathbb{R}^2$ as well. Thus Taylor's theorem tells us that
\begin{equation}\label{eqn:TaylorThm}
			W(x,y) = \sum_{0\leq i+j\leq r} \left(\frac{\partial^{i+j} W}{\partial x^i\partial y^{j}}(0,0) \cdot\frac{ x^iy^{j}}{i!j!} \right)+ R_r(x,y), 
\end{equation}
		where the remainder term $R_r(x,y)$ is bounded in absolute value by
\begin{equation}\label{eqn:TaylorRemainder}
		|R_r(x,y)| \leq  C \cdot \max(x,y)^{r+1}
\end{equation}
		for some constant $C = C(W,r)$. (Note: we may obtain a more precise bound on the remainder as a function of $(x,y)$, but the bound above will be sufficient here.) It is also well known that $f^{(n)}(0) = 0$ for all $n\in\mathbb{N}$; thus for all $i,j\geq 1$, 
		\[
			\frac{\partial^{i+j} W}{\partial x^i\partial y^{j}}(0,0) = \left(-f^{(i)}(0)\right)\cdot\left(-f^{(j)}(0)\right) = 0.
		\]
		In fact, if either $i\geq 1$ or $j\geq 1$, this will hold. So the only nonzero term of the sum in (\ref{eqn:TaylorThm}) is \[\frac{\partial^{0} W}{\partial x^0\partial y^0}(0,0) = W(0,0) = 1.\]
		Therefore, (\ref{eqn:TaylorThm}) becomes 
		\begin{equation*}
		W(x,y) = 1+ R_r(x,y).
		\end{equation*}
		Now recall that \[U_r(x,y) = (1-x^r)(1-y^r) = 1-(x^r+y^r - x^ry^r).\]
		For any $(x,y)\in [0,1]^2$, we have $x^r+y^r - x^ry^r\geq 0$. So in order to show that $W$ is bounded below by $U_r$ in some neighborhood of $(0,0)$, it will be sufficient to show that $|R_r(x,y)|\leq x^r+y^r - x^ry^r$ for $(x,y)$ in the same neighborhood. And observe that
		\begin{align*}
		x^r+y^r-x^ry^r = x^r+y^r(1-x^r) \geq x^r.
		\end{align*}
		Similarly, $x^r+y^r-x^ry^r \geq y^r$; thus 
		\begin{align}\label{eqn:UBound}
		x^r+y^r-x^ry^r  \geq \max(x,y)^r.
		\end{align}
		We may combine this with the bound on $|R_r(x,y)|$ given by \ref{eqn:TaylorRemainder} after making one last observation: for any constant $C=C(r,W)$, if $(x,y)$ is sufficiently close to $(0,0)$, then $C\cdot \max(x,y) \leq 1$. Therefore, for $(x,y)$ sufficiently close to $(0,0)$, combining \ref{eqn:TaylorRemainder} and \ref{eqn:UBound}, we obtain
		\begin{align*}
			x^r+y^r-x^ry^r  &\geq \max(x,y)^r \\
			&\geq C\max(x,y)\cdot \max(x,y)^r\\
			&\geq |R_r(x,y)|.
		\end{align*}
		Thus, as argued above, 
		\begin{align*}
		W(x,y)\geq U_r(x,y)
		\end{align*}		
	for $(x,y)$ in some neighborhood of $(0,0)$. Therefore, we may apply Lemma~\ref{lem:LocalDominance}, and conclude that 
\begin{align*}
				\omega(\mathbb{G}(n,W))&\geq (1-o(1))\cdot \omega(\mathbb{G}(n,U_r)) - O(\log n)\\
				& \geq (1-o(1))\cdot \tfrac{1}{2} \cdot e^{-\frac{2}{1+r}}\cdot n^{\frac{r}{r+1}},
\end{align*}
a.a.s., where the last line is the lower bound on $\omega(\mathbb{G}(n,U_r))$ from Lemma~\ref{lem:PolyLowerBd}. Then, since $r$ can be chosen to be arbitrarily large, we obtain
\[
	\omega(\mathbb{G}(n,W)) = n^{1-o(1)}
\]
a.a.s., as desired. 
	\end{proof}
	
	\section{Graphons equal to 1 at infinitely many points}\label{sec:Extensions}
	
	In this section, we prove Proposition~\ref{prop:LineExample}, and discuss other directions in which this work could be extended. We have described the clique number of a wide variety of $W$-random graphs where $W(a,a)=1$ for a finite number of $a\in[0,1]$. We could also ask for some characterization of clique numbers of $W$-random graphs when $W(a,a) = 1$ at an infinite number of points, either countable or uncountable. For example, what is the clique number of $\mathbb{G}(n,W)$ for the following graphon $W$?
	
	\begin{example}\label{ex:Countable}
		Let $W(x,y) =\big(1-x\sin^2{\big(\frac{1}{x}\big)}\big)\cdot \big(1-y\sin^2{\big(\frac{1}{y}\big)}\big)$.
	\end{example}
	
	In this case, we have $W(a,a)=1$ at a countably infinite number of points, namely for all $a$ with $\frac{1}{a} = k\cdot \pi$ for $k\in \mathbb{N}$. If we define $W(0,0) = 1$, we may also show that $W$ is locally Lipschitz at $(0,0)$, giving $\omega(\mathbb{G}(n,W)) = \Omega(\sqrt{n})$. The upper Dini derivatives of $W$ at $(0,0)$ are $0$, however, so we cannot use Lemma~\ref{lem:UpperDini} to give an upper bound. It could be interesting to find the correct order of growth of the clique number for this and other examples with a countably infinite number of points with $W(a,a) = 1$.
	
	Proposition~\ref{prop:LineExample} (restated here) gives a rough estimate of the order of growth of $\omega(\mathbb{G}(n,W))$ for a graphon $W$ with $W(a,a) = 1$ on an interval; the following graphon is  equal to 1 along the line $x=y$ and drops off away from that line.
	
	\LineExample*
	
	Before proving this proposition, let us note one difficulty in analyzing this and other graphons that are equal to 1 on a positive-measure portion of the line $x=y$. Namely, to obtain an upper bound on the clique number of such a graphon $W$, we will not easily be able to use the first moment method as with $W_r$ and $U_r$ in Sections \ref{sec:Sqrt} and \ref{sec:Poly}. In order to do so, we would need to compute
	\[
	\mathbb{E}[X_k] = \binom{n}{k}\int_{[0,1]^{k}}\prod_{\ell\neq m\in[k]}W(x_\ell,x_m)\ d\vec{x},
	\]
	where $X_k$ is the number of cliques in $\mathbb{G}(n,W)$ of size $k$. For $W_r$ and $U_r$, we were able to simplify this integral by using the fact that $W_r(x,y)$ and $U_r(x,y)$ are of the form $f(x)f(y)$ for some function $f$. Graphons of this form are called ``rank-1", and we can think of the $W$-random graphs that they produce as a more limited generalization of Erd\H{o}s-R\'enyi random graphs than those produced by graphons generally; in a rank-1 graphon, edge probabilities are not fixed as in the Erd\H{o}s-R\'enyi model, but the likelihood of each pair of vertices to be connected by an edge is determined only by how well-connected these vertices are overall, and not on any more complicated relationship between vertex weights.
	
	The graphon $W$ in the proposition above is not rank-1, so we cannot simplify the first moment calculation above by the same method we used for $W_r$ and $U_r$. More generally, any rank-1 graphon that is equal to 1 on some positive-measure portion of the line $x=y$ is in some sense trivial; if we have a graphon $W$ with $W(x_\ell,x_m) = f(x_\ell)f(x_m)$ and $W(a,a) = 1$ for all $a$ in some positive-measure $A\subseteq[0,1]$, then $f(a) = 1$ for $a\in A$. This would imply that $W$ evaluates to 1 on the positive-measure set $A\times A$, and thus $\mathbb{G}(n,W)$ has a linear-size clique number. So among graphons that are equal to 1 on some positive-measure portion of the line $x=y$, we are primarily interested in those that are not rank-1, and are therefore not susceptible to the simplified first moment calculation technique that we used for $W_r$ and $U_r$.
	
	Instead, to obtain the rough order of growth of $\mathbb{G}(n,W)$ for $W$ in Proposition \ref{prop:LineExample}, we will use a more direct approach; we expect that any set of vertices forming a large clique in $\mathbb{G}(n,W)$ would be sampled from a relatively small interval, as two vertices $x_i$ and $x_j$ are only likely to be connected in $\mathbb{G}({n,W})$ if $|x_i-x_j|$ is small. However, Lemma \ref{lem:IntervalOccupied} tells us that a.a.s.\ there will be no very large set of vertices sampled from a very small interval. We then take a union bound over all sufficiently large sets of vertices (which must each be spread over a not-too-small interval) to show that a.a.s.\ we will not obtain a ``large" clique. Following are the details of that argument.
	
	\begin{proof}[Proof of Proposition~\ref{prop:LineExample}]
		First, observe that $W$ is locally Lipschitz at, for example, the point $(0,0)$; all directional derivatives exist there and are bounded between $-1$ and $0$. So by Lemma~\ref{lem:Lipschitz}, $\mathbb{G}({n,W}) = \Omega(\sqrt{n})$ a.a.s. Now we compute an upper bound on the clique number, using the method outlined in the previous paragraph.
		
		Consider any set $S$ of $k=3\delta  n$ vertices in $\mathbb{G}(n,W)$, with $\delta= \omega\big(\frac{1}{\sqrt{n}}\big)$ to be chosen later; we wish to show that no such set will form a clique. Partition $S$ into $S_1,S_2$, and $S_3$, namely the first $\delta n$ vertices, the middle, and the last, respectively, as they are ordered on the unit interval. By Lemma \ref{lem:IntervalOccupied}, with probability $1-o(1)$, the vertices in each set, and in particular in $S_2$, occupy an interval of length at least $\frac{\delta}{2}(1-o(1))$. Therefore each vertex in $S_1$ is at distance at least $\frac{\delta}{2}(1-o(1))$ from each vertex in $S_3$, and hence by the definition of W, the probability that every such pair of vertices is connected is at most
		\[
		\left(1-\frac{\delta}{2}(1-o(1))\right)^{(\delta n)^2},
		\]
		which gives an upper bound on the probability that $S$ is a clique. Taking a union bound over all sets of $3\delta n$ vertices in $\mathbb{G}(n,W)$, the probability that there exists a clique of size $k=3\delta n$ in $\mathbb{G}(n,W)$ is at most
		\begin{align*}
		\binom{n}{3\delta n}\left(1- \frac{\delta}{2}(1-o(1))\right)^{(\delta n)^2}
		&\leq \left(\frac{en}{3\delta n}\right)^{3\delta n}\left(1-\frac{\delta}{2}(1-o(1))\right)^{(\delta n)^2}\\
		&\leq e^{3\delta n\cdot\log\frac{e}{3\delta}}\ \cdot\ e^{-(\delta n)^2\frac{\delta}{2}(1-o(1))}\\
		&= e^{\delta n\left(3\log\frac{e}{3\delta}-\frac{\delta^2n}{2}(1-o(1)) \right)}.
		\end{align*}
		This will be $o(1)$ if $3\log\frac{e}{3\delta}\leq\frac{\delta^2n}{2}(1-\Omega(1))$, which is satisfied, for example, for $\delta = \frac{1}{\sqrt{n}}\log n = n^{-1/2+o(1)}$ (but not, say, for $\delta = \frac{1}{\sqrt{n}}(\log n)^{1/4}$). So with probability $1-o(1)$, the clique number of  $\mathbb{G}(n,W)$ is at most $3\delta n = n^{-1/2+o(1)}\cdot n = n^{1/2+o(1)}$.
	\end{proof}
	
	\appendix\section{Variance in number of cliques}
	
	In this section, we show that the numbers of $k$-cliques in $\mathbb{G}(n,W_r)$ and $\mathbb{G}(n,U_r)$ have high variance for $k$ within a reasonable range (Lemma~\ref{cor:BigVarSqrtAndPoly}). This makes it impossible to directly use the second moment method to find a useful lower bound on the clique number of these graphs.
	
	In more detail, our setting is as follows: given any graphon $W$, we will write $X_k$ for the number of $k$-cliques in $\mathbb{G}(n,W)$. Suppose that, for a given graphon $W$, we have found a cutoff value $k=k(n)$ at which $\mathbb{E}[X_k]$ goes from asymptotically infinite to asymptotically zero, giving an upper bound of $\omega(\mathbb{G}(n,W)) \leq (1+o(1))k$ with probability $1-o(1)$ by Markov's inequality. In order to prove a matching lower bound, we would like to show that the number of cliques of size $(1-o(1))k$ in $\mathbb{G}(n,W)$ is a.a.s.\ nonzero. Perhaps the simplest way to do this, and the technique used for Erd\H{o}s-R\'enyi random graphs in \cite{GrMc75} and \cite{Ma76},  is the second moment method; namely, Chebyshev's inequality gives the bound
	\begin{equation}\label{eqn:Chebyshev}
	\Pr[X_k = 0] \leq \frac{\Var(X_k)}{\Ex[X_k]^2} = \frac{\Ex[X_k^2]}{\Ex[X_k]^2}-1
	\end{equation}
	for any $k$. If $\Var(X_k) = o(\Ex[X_k]^2)$, or equivalently $\Ex[X_k^2]=(1+o(1))\Ex[X_k]^2$, then this shows that $X_k\geq 1$ with probability $1-o(1)$, and thus $\omega(\mathbb{G}(n,W)) \geq k$ a.a.s.
	
	The entire challenge of applying the second moment method lies in obtaining a good bound on the ratio ${\Ex[X_k^2]}/{\Ex[X_k]^2}$. The following lemma gives a slightly more explicit expression for this quantity; it is a standard result adapted slightly for this application (see Sections 4.3 and 4.5 of \cite{AlSp04}).
	
	\begin{lemma}\label{lem:VarBound}
		Let $W$ be a graphon, and for $S\subseteq [n]$, let $A_S$ be the event that the elements of $S$ form a clique in $\mathbb{G}(n,W)$. Then
		\[
		\frac{\Ex[X_k^2]}{\Ex[X_k]^2} = \sum_{i=0}^{k} \frac{\binom{k}{i} \binom{n-k}{k-i}}{\binom{n}{k}}\cdot \frac{ \Pr[A_{S_i}\cap A_{[k]}]}{\Pr[A_{[k]}]^2},
		\]
		where $S_i$ is any subset of $[n]$ of size $k$ that intersects $[k]$ in exactly $i$ elements.
	\end{lemma}
	
	\begin{proof}
		This lemma follows from a direct computation of the first and second moments; first, write $$X_k = \displaystyle\sum_{S\subseteq [n],\, |S| = k} I_S,$$
		where $I_S$ is the indicator variable for the vertices in $S$ forming a clique. With this notation, we obtain
		\begin{align}\label{eqn:FirstMom}
					\Ex[X_k] = \sum_{S\subseteq [n],\, |S| = k} \Ex[I_S] = \binom{n}{k}\Pr[A_{[k]}]. 
		\end{align}
		Similarly, 
		\begin{align}\label{eqn:SecondMom}
			\Ex[X_k^2] = \sum_{\substack{S, T\in [n]\\\ |S|, |T| = k}} \Ex\left[I_SI_T\right] = \sum_{\substack{S, T\in [n]\\\ |S|, |T| = k}}\Pr[A_S\cap A_T].
		\end{align}
		And notice that this last probability depends only on the size of the intersection of $S$ and $T$; thus we can group the terms of the sum above by the size $i$ of the intersection. The number of ways to choose two sets of $k$ vertices that overlap in exactly $i$ elements is $\binom{n}{k}\binom{k}{i}\binom{n-k}{k-i}$; so [\ref{eqn:SecondMom}] becomes
		\begin{align}\label{SecondMom2}
			\Ex[X_k^2] = \sum_{i=0}^k\binom{n}{k}\binom{k}{i}\binom{n-k}{k-i}\Pr[A_{[k]}\cap A_{S_i}].
		\end{align}
		And combining [\ref{eqn:FirstMom}] and [\ref{eqn:SecondMom}], we see that
		\[
		\frac{\Ex[X_k^2]}{\Ex[X_k]^2} = \sum_{i=0}^{k} \frac{\binom{k}{i} \binom{n-k}{k-i}}{\binom{n}{k}}\cdot \frac{ \Pr[A_{S_i}\cap A_{[k]}]}{\Pr[A_{[k]}]^2},
		\]
		as desired.
	\end{proof}

	Now, in order to apply these results to any graphon $W$, we need to compute the sum given in the lemma above, and in particular, $\frac{\Pr[A_{S_i}\cap A_{[k]}]}{\Pr[A_{[k]}]^2}$. For $W$ of the right form, we can obtain a more explicit expression:
	
	\begin{lemma}\label{lem:SepGraphonSndMom}
		For any graphon $W$ of the form $W(x,y) = f(x)f(y)$, i.e., for any $W$ that is rank-1,
		\[
		\frac{\Pr[A_{S_i}\cap A_{[k]}]}{\Pr[A_{[k]}]^2} = \left(\int_0^1 f(x)^{k-1}dx\right)^{-2i}\cdot \left(\int_0^1 f(x)^{2k-i-1}dx\right)^{i},
		\]
		where $A_S$ is the event that the elements of $S$ form a clique in $\mathbb{G}(n,W)$, and $S_i$ is any subset of $[n]$ of size $k$ that intersects $[k]$ in exactly $i$ elements.
	\end{lemma}
	
	\begin{proof}
		We begin by computing $\Pr[A_{S_i}\cap A_{[k]}]$, considering in three parts the edges of the graph consisting of a clique on $[k]$ and a clique on $S_i$. This is equal to
		\begin{align*}
		\Pr[A_{S_i}\cap A_{[k]}]
		&=\int_{[0,1]^{2k-i}}\Bigg(\prod_{\substack{\ell\neq m\in S\setminus (S\cap [k])\\\text{ or } [k]\setminus (S\cap [k])}}f(x_\ell)f(x_m)\Bigg)\cdot\Bigg(\prod_{\ell\neq m\in  S\cap[k]}f(x_\ell)f(x_m)\Bigg)\\
		&\ \ \ \ \ \ \ \ \ \ \ \ \ \ \cdot\Bigg(\prod_{\substack{\ell\in S\cap[k],\\ m \in (S\cup [k])\setminus (S\cap [k])}}f(x_\ell)f(x_m)\Bigg)d\vec{x}\\
		&=\int_{[0,1]^{2k-i}}\Bigg(\prod_{\ell\in (S\cup [k])\setminus (S\cap [k])}f(x_\ell)^{k-1}\Bigg)\cdot \Bigg(\prod_{\ell\in S\cap [k]}f(x_\ell)^{2k-i-1}\Bigg)d\vec{x}\\
		&=\left(\int_0^1 f(x)^{k-1}\,dx\right)^{2k-2i}\cdot \left(\int_0^1 f(x)^{2k-i-1}\,dx\right)^{i}
		\end{align*}
		Without any further assumptions on $f(x)$, this is as far as $\Pr[A_{S_i}\cap A_{[k]}]$ can be evaluated. To finish off, we compute
		\begin{align*}
		\Pr[A_{[k]}]& = \int_{[0,1]^{k}}\prod_{\ell\neq m\in[k]}f(x_\ell)f(x_m)\,d\vec{x}\\
		& = \left(\int_0^1 f(x)^{k-1}\,dx\right)^k.
		\end{align*}
		Therefore
		\begin{align*}
		\frac{\Pr[A_{S_i}\cap A_{[k]}]}{\Pr[A_{[k]}]^2}
		&=\left(\int_0^1 f(x)^{k-1}dx\right)^{-2i}\cdot \left(\int_0^1 f(x)^{2k-i-1}dx\right)^{i}.
		\end{align*}
	\end{proof}
	
	For the graphons $W_r$ and $U_r$, we can evaluate the integrals above and obtain more explicit expressions:
	
	\begin{lemma}\label{lem:SndMomSqrtAndPoly}
		Given any $k = \omega(1)$ and $1\leq i\leq k-1$,
		\begin{enumerate}[(i)]
			\item\label{lem:SndMomSqrt} for the graphon $W_r$,
			\[  
			\frac{\Pr[A_{S_i}\cap A_{[k]}]}{\Pr[A_{[k]}]^2}  = \left(\Theta(k)\right)^i,
			\]
			\item\label{lem:SndMomPoly} and for the graphon $U_r$,
			\[  
			\frac{\Pr[A_{S_i}\cap A_{[k]}]}{\Pr[A_{[k]}]^2}  = \left(\Theta(k^{1/r})\right)^i.
			\]
		\end{enumerate}
	\end{lemma}
	\begin{proof}
		We begin with (\ref{lem:SndMomSqrt}). For the graphon $W_r(x,y) = (1-x)^r(1-y)^r$, by Lemma~\ref{lem:SepGraphonSndMom}, 
		\begin{align*}
		\frac{\Pr[A_{S_i}\cap A_{[k]}]}{\Pr[A_{[k]}]^2} &=\left(\int_0^1 (1-x)^{r(k-1)}dx\right)^{-2i}\cdot \left(\int_0^1 (1-x)^{r(2k-i-1)}dx\right)^{i}\\
		&=    \left(\frac{1}{r(k-1)+1}\right)^{-2i}\left(\frac{1}{r(2k-i-1)+1}\right)^{i}\\
		& = (\Theta(k))^i.
		\end{align*}
		Now we prove (\ref{lem:SndMomPoly}). Again by Lemma~\ref{lem:SepGraphonSndMom}, for $U_r(x,y) = (1-x^r)(1-y^r)$, we have
		\begin{align*}
		\frac{\Pr[A_{S_i}\cap A_{[k]}]}{\Pr[A_{[k]}]^2} &=\left(\int_0^1 (1-x^r)^{k-1}dx\right)^{-2i}\cdot \left(\int_0^1 (1-x^r)^{2k-i-1}dx\right)^{i}.
		\end{align*}
		And as computed in the proof of Lemma~\ref{lem:PolyUpperBd},
		\[
		\int_0^1 (1-x^r)^{k-1}dx = \frac{\Gamma(k)\cdot \Gamma(1+\frac{1}{r})}{\Gamma(k+\frac{1}{r})}.
		\]
		Applying this to the expression above, we obtain
		\begin{align}
		\frac{\Pr[A_{S_i}\cap A_{[k]}]}{\Pr[A_{[k]}]^2} &=
		\left(\frac{\Gamma(2k-i)\Gamma(1+\frac{1}{r})}{\Gamma(2k-i+\frac{1}{r})}\right)^{i}
		\left(\frac{\Gamma(k)\Gamma(1+\frac{1}{r})}{\Gamma(k+\frac{1}{r})}\right)^{-2i}. \label{eqn:ConditionalGammas}
		\end{align}
		Using the approximation $\frac{\Gamma(k)}{\Gamma(k+\frac{1}{r})} =k^{-\frac{1}{r}}(1+o(1))$ obtained from Stirling's formula, (\ref{eqn:ConditionalGammas}) becomes
		\begin{align*}
		\frac{\Pr[A_{S_i}\cap A_{[k]}]}{\Pr[A_{[k]}]^2} = &
		\left((2k-i)^{-\frac{1}{r}}\cdot \Gamma(1+\tfrac{1}{r})(1+o(1))\right)^{i}
		\left(k^{-\frac{1}{r}}\cdot \Gamma(1+\tfrac{1}{r})(1+o(1))\right)^{-2i}\\
		=&
		\left(\frac{k^{2/r}}{\Gamma(1+\frac{1}{r})(2k-i)^{1/r}}(1+o(1))\right)^i\\
		=&
		\left(\Theta(k^{1/r})\right)^i.\qedhere
		\end{align*}
	\end{proof}
	
	We are now nearly ready to show that for the graphons $W_r$ and $U_r$, and for any reasonably large $k$, the number of $k$-cliques in $\mathbb{G}(n,W)$ has large variance.
	
	\begin{theorem}\label{thm:GeneralBigVar}
		For any $r>0$ and any graphon $W$, if $$\frac{\Pr[A_{S_i}\cap A_{[k]}]}{\Pr[A_{[k]}]^2} = \left(\Omega(k^{1/r})\right)^i,$$
		then for any $k = \Theta\left(n^{\frac{r}{r+1}}\right)$, we have
		$\Var(X_k) = \omega(\Ex[X_k]^2).$
	\end{theorem}
	
	Before proving the theorem, note that together with Lemma~\ref{lem:SndMomSqrtAndPoly}, it directly implies the following corollary.
	
	\begin{corollary}\label{cor:BigVarSqrtAndPoly}
		Given $r>0$,
		\begin{enumerate}[(i)]
			\item for any $k=\Theta(\sqrt{n})$, if $X_k$ is the number of $k$-cliques in $\mathbb{G}(n,W_r)$, then $\Var(X_k)=\omega(\mathbb{E}[X_k]^2)$, and \label{cor:BigVarSqrt}
			\item for any $k=\Theta\left(n^{\frac{r}{r+1}}\right)$, if $X_k$ is the number of $k$-cliques in $\mathbb{G}(n,U_r)$, then $\Var(X_k)=\omega(\mathbb{E}[X_k]^2)$. \label{cor:BigVarPoly}
		\end{enumerate}
	\end{corollary}
	Now we prove the theorem.
	\begin{proof}[Proof of Theorem~\ref{thm:GeneralBigVar}]
		We will apply Lemma~\ref{lem:VarBound} to show that ${\Ex[X_k^2]}/{\Ex[X_k]^2} = \omega(1)$, or equivalently, $\Var(X_k) = \omega(\Ex[X_k]^2)$. Recall that, by Lemma~\ref{lem:VarBound} and by hypothesis,
		\[
		\frac{\Ex[X_k^2]}{\Ex[X_k]^2} = \sum_{i=1}^{k-1} \frac{\binom{k}{i} \binom{n-k}{k-i}}{\binom{n}{k}}\cdot \left(\Omega(k^{1/r})\right)^i.
		\]
		We will show not only that this sum is $\omega(1)$, but in fact, that it always contains a term that is $\omega(1)$. This comes down almost entirely to appropriately estimating the three binomial coefficients appearing in the $i^\text{th}$ term of the sum above. First, for any $k$ that is $\omega(1)$ but sublinear, 
		\begin{equation}\binom{n}{k}
		=\left(\frac{ne}{k}\right)^ke^{-o(k)}\label{eqn:FirstBinomEst}.
		\end{equation}
		Next, observe that for all $0\leq i\leq k$, since $k=o(n)$, we also have $k-i = o(n-k)$. If $i=\eps k$ for some constant $0<\eps<1$, then $(k-i) = \omega(1)$ as well, and we obtain
		\begin{align}
		\binom{n-k}{k-i} &= \left(\frac{(n-k)e}{k-i}\right)^{k-i}e^{-o(k-i)}\notag\\
		&\geq\left(\frac{ne}{k}\right)^{(1-\eps)k}e^{-o(k)}\label{eqn:SecondBinomEst}.
		\end{align}
		We also have
		\begin{equation}\label{eqn:ThirdBinomEst}
		\binom{k}{i} \geq \left(\frac{k}{i}\right)^i = e^{\eps k\log \frac{1}{\eps}}.
		\end{equation}
		Together, (\ref{eqn:FirstBinomEst}), (\ref{eqn:SecondBinomEst}), and (\ref{eqn:ThirdBinomEst}) imply that for $i=\eps k = \Theta(k)$, the $i^\text{th}$ term of the sum above is
		\begin{align*}
		\frac{\binom{k}{i} \binom{n-k}{k-i}}{\binom{n}{k}}\cdot \left(\Omega(k^{1/r})\right)^i
		& \geq \frac{e^{\eps k\log \frac{1}{\eps}}\cdot \left(\frac{ne}{k}\right)^{(1-\eps)k}e^{-o(k)}}{\left(\frac{ne}{k}\right)^ke^{-o(k)}}\cdot \left(\Theta(k^{1/r})\right)^i \\
		& = e^{\eps k\log \frac{1}{\eps}-o(k)} \left(\frac{ne}{k}\right)^{-\eps k}\cdot \left(\Theta\left(k^{1/r}\right)\right)^{\eps k} \\
		&=e^{\eps k\log \frac{1}{\eps}-o(k)} \left(\frac{ne}{k}\right)^{-\eps k}\cdot \left(\Theta\left(\frac{n}{k}\right)\right)^{\eps k}, && \text{since }k = \Theta\left(n^{\frac{r}{r+1}}\right) \\
		&=e^{\eps k(\log \frac{1}{\eps}-C) -o(k)}
		\end{align*}
		for some constant $C$.
		Note that we can make $C$ as large as we want by controlling the size of the implicit constant in $k=\Theta(n^\frac{r}{r+1})$. However, for any fixed choice of $C$, we can find some small but constant $\varepsilon = \varepsilon(C)$ such that $\log(1/\varepsilon)>\log(C)$. So for some $\varepsilon$, this expression will always be $\omega(1)$. Therefore ${\Ex[X_k^2]}/{\Ex[X_k]^2} = \omega(1)$, or equivalently  $\Var(X_k) = \omega(\Ex[X_k]^2)$, as desired.
	\end{proof}

	\section*{Acknowledgments}
	
	The author would like to thank Nike Sun, Remco van der Hofstad, and Yufei Zhao for helpful discussions, and Henry Cohn for his many helpful suggestions and for his unwavering support and encouragement.
	
	\bibliographystyle{amsplain}
	\bibliography{GraphonClique}
	
\end{document}